\documentclass[11pt]{article}

\usepackage{amsmath, amssymb, amsthm}
\usepackage{tikz}

\theoremstyle{plain}
\newtheorem{thm}{Theorem}[section]
\newtheorem*{thm*}{Theorem}
\newtheorem{prop}[thm]{Proposition}
\newtheorem{lem}[thm]{Lemma}
\newtheorem*{lem*}{Lemma}

\newtheorem{claim}[thm]{Claim}

\newcommand{\codeg}{\text{codeg}}
\newcommand{\BBE}{\mathbb{E}}
\newcommand{\BFP}{\mathbf{P}}

\date{}
\title{\vspace{-0.7cm}Embedding degenerate graphs of small bandwidth}

\author{
Choongbum Lee \thanks{Department of Mathematics,
MIT, Cambridge, MA 02139-4307. Email: cb\_lee@math.mit.edu.
Research supported by NSF Grant DMS-1362326.}
}

\begin{document}

\maketitle

\begin{abstract}
We develop a tool for embedding
almost spanning degenerate graphs of small bandwidth.
As an application, we extend the blow-up lemma to degenerate graphs
of small bandwidth, the bandwidth theorem to degenerate graphs,
and make progress on
a conjecture of Burr and Erd\H{o}s on Ramsey number of degenerate graphs.
\end{abstract}

\section{Introduction} \label{sec:intro}

An embedding of a graph $H$ into a graph $G$ is an injective map $f : V(H) \rightarrow V(G)$
for which $\{f(v), f(w)\}$ is an edge of $G$ whenever $\{v,w\}$ is an edge of $H$.
The study of sufficient conditions which force the existence of an embedding of $H$ into $G$
is a fundamental topic in graph theory that has been studied from various different aspects.
For example, Tur\'an's theorem is a foundational result that
establishes the minimum number of edges needed to guarantee an embedding
of a complete graph into another graph.

Another fundamental result is Dirac's theorem \cite{Dirac} which asserts 
that every $n$-vertex graph $G$
of minimum degree at least $\frac{n}{2}$ contains a Hamilton cycle, i.e., a 
cycle passing through every vertex of the graph exactly once.
Dirac's theorem influenced the development of the study of embedding large subgraphs 
(graphs whose number of vertices have the same order of magnitude as that of the host graph),
a classical example being Hajnal and Szemer\'edi's theorem \cite{HaSz} asserting that every graph of minimum degree
at least $(1- \frac{1}{r})n$ contains $\lfloor \frac{n}{r} \rfloor$ vertex-disjoint copies 
of $K_r$. 

In 1997, Koml\'os, S\'ark\"ozy, and Szemer\'edi  \cite{KoSaSz} made a major breakthrough in this direction. 
They developed the powerful blow-up lemma and used it to tackle various problems such as 
Pos\'a-Seymour conjecture on power of Hamilton cycles \cite{KoSaSz1}, and 
Alon-Yuster conjecture on $F$-packing \cite{KoSaSz2}.
This line of research culminated in the so called Bandwidth Theorem of 
B\"ottcher, Schacht, and Taraz \cite{BoScTa}.
The {\it bandwidth} of a graph $H$ is the minimum integer $b$ for which 
there exists a labelling of its vertices by integers $1,2,\ldots, |V(H)|$
where $|i-j| \le b$ holds for every edge $\{i, j\}$.
Confirming a conjecture of Bollob\'as and Koml\'os, they proved that 
for every integer $r$ and positive real $\delta$, there exists $\beta$ such that
every $n$-vertex graph $G$ (with large $n$) of minimum degree at least 
$(1 - \frac{1}{r} + \delta)n$ contains all $n$-vertex $r$-chromatic 
graphs of bandwidth at most $\beta n$ as a subgraph.

As explained above, the key tool used in these proofs is the blow-up lemma.
A bipartite graph with parts $A \cup B$ is {\it $\varepsilon$-regular} if for every 
$X \subseteq A$ and $Y \subseteq B$ of sizes $|X| \ge \varepsilon|A|$ and $|Y| \ge \varepsilon|B|$, 
the relation $\left| \frac{e(X,Y)}{|X||Y|} - \frac{e(A,B)}{|A||B|} \right| \le \varepsilon$ holds.
Define the {\it relative minimum degree} of a bipartite graph with parts $A \cup B$ as
the minimum of the two quantities $\min_{a \in A} \frac{d(a)}{|B|}$ and $\min_{b \in B} \frac{d(b)}{|A|}$.
In the simplest case for bipartite graphs, the blow-up lemma asserts that 
for every $\delta$ and $\Delta$, there exists $\varepsilon$ such that the following holds:
if $G$ is a bipartite graph with parts $A \cup B$ for which $(A,B)$ is $\varepsilon$-regular
with relative minimum degree at least $\delta$,
then it contains all bipartite graphs $H$
with parts $X \cup Y$ of sizes $|X| \le |A|$ and $|Y| \le |B|$ of maximum degree
at most $\Delta$ as a subgraph. 

Another stream of research in the study of embedding large subgraphs was 
developed in Ramsey theory.
The Ramsey number of a graph $H$, denoted by $r(H)$, is the minimum integer $n$
for which every two-coloring of the edge set of $K_n$, the complete graph on 
$n$ vertices, contains a monochromatic 
copy of $H$. The study of Ramsey numbers was initiated in the seminal paper of 
Ramsey \cite{Ramsey} and now became one of the most active areas of research
in combinatorics. 

One important problem in Ramsey theory is to understand the asymptotics
of $r(H)$ for sparse graphs $H$. 
Since establishing small upper bounds on Ramsey number can be rephrased
as finding large monochromatic subgraphs in edge-colored graphs, it is
somewhat inevitable that developments in the two problems are closely related to each other.
A graph is {\it $d$-degenerate} if there exists an ordering of its vertex set
in which each vertex has at most $d$ neighbors preceding itself.
Degeneracy is a natural measure of sparseness since it implies that every 
subset of $m$ vertices span at most $md$ edges.
The phenomenon of sparse graphs having small Ramsey number was first anticipated 
and studied by Burr and
Erd\H{o}s \cite{BuEr} who made the following influential conjecture (and offered
a cash prize for a solution):
for every $d \ge 1$, there exists a constant $c = c(d)$ such that every $n$-vertex
$d$-degenerate graph $H$ satisfies $r(H) \le c(d) n$.
The conjecture is still open and the current best bound is given by 
Fox and Sudakov \cite{FoSu09, FoSu09-2} who, improving the result of 
Kostochka and Sudakov \cite{KoSu}, proved that there exists a constant $c_d$ depending
on $d$ such that $r(H) \le 2^{c_d\log^{1/2}n}n$ holds for every
$d$-degenerate $n$-vertex graph $H$.

Some versions of this conjecture with relaxed sparsity conditions have been established.  
The case when $H$ has bounded maximum degree was solved by Chv\'atal, R\"odl, 
Szemer\'edi, and Trotter \cite{ChRoSzTr} using the regularity lemma, and the bound
was later improved in subsequent papers \cite{Eaton, GrRoRu00, FoSu09, CoFoSu}. 
An $m$-vertex graph $H$ is {\it $a$-arrangeable} if its vertices can be labelled by $[m] = \{1,2,\cdots, m\}$ so that
$|N^{-}(N^{+}(i))| \le a$ holds every vertex $i \in [m]$ (where $N^{-}(x)$ is the set of
neighbors of $x$ that have smaller label than $x$, and $N^{+}(x)$ is the set of 
neighbors of $x$ that have larger label than $x$).
Chen and Schelp \cite{ChSc} introduced the concept of arrangeability
in their study of Burr and Erd\H{o}s's conjecture and proved that
the conjecture holds for graphs with bounded arrangeability.
Note that arrangeability is a measure of sparseness of graphs that lies strictly between
bounded degree and bounded degeneracy. For example,
if a graph contains a vertex $v$ in at least $\frac{3a+1}{2}$ triangles
sharing no other vertex than $v$, then it has arrangeability
at least $a$, but may be $2$-degenerate. Another example is a 2-subdivision of a complete graph,
a graph obtained from the complete graph by replacing each edge with
internally vertex-disjoint paths of length 2.

\medskip

In this paper, we develop embedding results of graphs with bounded degeneracy and small bandwidth.
The main theorem is an almost-spanning blow-up lemma for such graphs.
For a graph $G$, a pair of vertex subsets $X$ and $Y$ are {\em $(\varepsilon, \delta)$-dense}
if for every $X' \subseteq X$ and $Y' \subseteq Y$ of sizes $|X'| \ge \varepsilon |X|$
and $|Y'| \ge \varepsilon |Y|$ the number of edges between $X'$ and $Y'$ is at least 
$\delta |X'||Y'|$.
An $r$-partite graph $G$ with vertex partition $V_1 \cup \cdots \cup V_r$ is 
$(\varepsilon,\delta)$-dense if $(V_i, V_j)$ is $(\varepsilon,\delta)$-dense
for all pairs of distinct indices $i,j \in [r]$.
The original blow-up lemma is a spanning embedding result of bounded degree graphs
into $(\varepsilon, \delta)$-dense graphs of large enough minimum degree.
Recently, B\"ottcher, Kohayakawa, Taraz, and W\"urfl \cite{BoKoTaWu}
showed that the condition of having bounded degree can be relaxed into the condition
of having bounded arrangeability and maximum degree at most $\frac{\sqrt{n}}{\log n}$.
The following theorem further relaxes the condition and asserts that an
almost-spanning blow-up lemma holds for degenerate graphs of small bandwidth.

\begin{thm} \label{thm:main}
For each positive integers $r, d$ and positive real numbers $\varepsilon, \varepsilon', \delta$,
satisfying $\varepsilon \le (\frac{\delta}{2})^{2r}$,
there exists $N$ such that the following holds for all $n \ge N$.
Let $H$ be an $r$-partite graph over a vertex partition $(W_i)_{i \in [r]}$ 
that is $d$-degenerate, has bandwidth at most $n^{1-\varepsilon'}$, and satisfies $|W_i| \le (1-\varepsilon)n$ for all $i \in [r]$.
Let $G$ be an $(\varepsilon^2,\delta)$-dense 
$r$-partite graph over a vertex partition $(V_i)_{i \in [r]}$ where $|V_i| = n$ for all $i \in [r]$.
Then $G$ contains $H$ as a subgraph.
\end{thm}

See Theorem~\ref{thm:main_extend} for a slightly more general form where we do not require
the parts to be of equal size, and give a more precise bound on the bandwidth.
Note that the complete bipartite graph $K_{d,n-d}$ is $d$-degenerate. 
In a random bipartite graph of density $\frac{1}{2}$ with $n$ vertices in each part,
almost surely every $d$-tuple of vertices have $(1+o(1))\frac{n}{2^d}$ common neighbors
and therefore it does not contain $K_{d,n-d}$ as a subgraph.
Hence the theorem above does not hold if we completely remove the restriction on bandwidth.
Nevertheless, it might be possible to replace the restriction by 
a moderate restriction such as the maximum degree being at most $c^{\Delta}n$ for some constant $c$.

The technique used in proving Theorem~\ref{thm:main} is different from that 
used in the previous versions of the blow-up lemmas.
Our embedding strategy has its origin in the methods developed in Ramsey theory.
More specifically, it is based on dependent random choice, a powerful technique in 
probabilistic combinatorics, and builds on the ideas developed by 
Kostochka and Sudakov \cite{KoSu}, and Fox and Sudakov \cite{FoSu09}.
This technique can further be utilized to 
extend the bandwidth theorem to almost spanning graphs of bounded degeneracy. 
An extension of the bandwidth theorem to arrangeable graphs was proved by
B\"ottcher, Taraz, and W\"urfl \cite{BoTaWu} using the version of blow-up lemma developed by
B\"ottcher, Kohayakawa, Taraz, and W\"urfl (their result requires a much more modest
bound of $o(n)$ on the bandwidth of $H$ compared to the $n^{1-o(1)}$ bound of our
theorem).

\begin{thm} \label{thm:main_bandwidth}
For all positive integers $r, d$ and positive real numbers $\varepsilon, \varepsilon', \delta$,
there exists $N$ such that the following holds for all $n \ge N$.
If $G$ is an $n$-vertex graph of minimum degree at least $(1 - \frac{1}{r} + \delta)n$,
and $H$ is an $r$-partite $d$-degenerate graph of bandwidth at most $n^{1-\varepsilon'}$
on at most $(1-\varepsilon)n$ vertices, then $G$ contains $H$ as a subgraph.
\end{thm}

Let $G$ be an $n$-vertex graph consisting of two cliques of order $(1-\frac{1}{r})n$
sharing $(1-\frac{2}{r})n$ vertices. Note that it contains two disjoint sets of
vertices each of size $\frac{1}{r}n$ which form an empty bipartite graph. 
Hence $r$-chromatic graphs with good expansion property cannot be embedded into $G$
and thus the bandwidth condition is necessary in Theorem~\ref{thm:main_bandwidth}.
The minimum degree condition can be relaxed for bipartite graphs.

\begin{thm} \label{thm:main_bipartite}
For all positive integers $d$ and positive reals $\varepsilon, \delta$,
there exists $N$ such that the following holds for all $n \ge N$.
If $G$ is an $n$-vertex graph of minimum degree at least $\delta n$,
and $H$ is a bipartite $d$-degenerate graph of bandwidth at most $e^{-100\sqrt{d \log(1/\varepsilon)\log n}}n$
on at most $(\delta-\varepsilon)n$ vertices, then $G$ contains $H$ as a subgraph.
\end{thm}

Since an $n$-vertex graph of density $\delta$ contains a subgraph of minimum degree at least 
$\frac{\delta}{2}n$, Theorem~\ref{thm:main_bipartite} gives a density-type embedding
theorem for bipartite degenerate graphs of small bandwidth.
Thus in some sense extends an embedding result of Fox and Sudakov \cite{FoSu09}
(improving that of Kostochka and Sudakov \cite{KoSu}), 
asserting the existence of a positive constant $c$ such that if $H$ is an 
$d$-degenerate bipartite graph on at most $e^{-c\sqrt{d \log(1/\delta)\log n}}n$ vertices,  
then every $n$-vertex graph of density at least $\delta$ contains a copy of $H$.

The technique also has applications in Ramsey theory.
Allen, Brightwell, and Skokan \cite{AlBrSk} proved that bounded degree
$n$-vertex $r$-chromatic graphs of bandwidth $o(n)$ has Ramsey number at most $(2r+4)n$.
We combine our technique with their framework to prove the following theorem
that brings us one step closer to the resolution of Burr and Erd\H{o}s's conjecture.

\begin{thm} \label{thm:ramsey_degenerate}
For all positive integers $r, d$ and positive real $\varepsilon$,
there exists $N$ such that the following holds for all $n \ge N$.
Let $H$ be an $n$-vertex $r$-chromatic $d$-degenerate graph of bandwidth at most $n^{1-\varepsilon}$.
Then $r(H) \le (2r+5)n$.
\end{thm}
 
Since a $d$-degenerate graph has chromatic number at most $d+1$, Theorem~\ref{thm:ramsey_degenerate}
implies that $r(H) \le (2d+7)n$ regardless of the chromatic number of $H$. 
Therefore it proves Burr and Erd\H{o}s's conjecture for
graphs of small bandwidth.
Furthermore, given a graph, we can add isolate vertices to obtain a graph with small bandwidth.
Hence a more precise form given in Theorem~\ref{thm:main_bandwidth_refine}
shows that Ramsey numbers of $d$-degenerate graphs are nearly linear. Such result
was previously obtained in \cite{FoSu09, KoSu} with similar but better bounds.

The phenomenon of degenerate graphs of small bandwidth having small Ramsey numbers
has been observed before in a slightly different context.
Let $\sigma(H)$ be the size of the smallest color class in any proper $\chi(H)$-coloring of the
vertices of $H$. A simple construction observed by Chv\'atal and Harary \cite{ChHa} (and
strengthened by Burr \cite{Burr}) shows that $r(G,H) \ge (\chi(H)-1)(|V(G)|-1) + \sigma(H)$ holds
for all graphs $H$ and connected graphs $G$. A graph $G$ is {\em $H$-good} if the equality holds.
Ramsey-goodness is an extensively studied topic in Ramsey theory (see, e.g., \cite{Burr, BuEr, NiRo, AlBrSk, CoFoLeSu, FiGrMoSa}).
The work most related to ours is that of Nikiforov and Rousseau \cite{NiRo}.
They showed that for each fixed $s$ and $d$, every sufficiently large 
$d$-degenerate $n$-vertex graph of bandwidth at most $n^{1-o(1)}$ is $K_s$-good
(to be more precise, their result requires the graph to have small separators instead of
small bandwidth). Their result establishes a bound on an off-diagonal Ramsey number, and Theorem~\ref{thm:ramsey_degenerate} can be seen as a corresponding diagonal result.

Note that there is (about) a factor of 2 difference between the Ramsey-goodness bound on $r(H)$ and the bound given in Theorem~\ref{thm:ramsey_degenerate} for large values of $r$. It is likely that 
the bound in Theorem~\ref{thm:ramsey_degenerate} can be improved to $r(H) \le (1+c_r)rn$,
where $c_r$ is a constant depending on $r$ that tends to zero as $r$ tends to infinity.
We remark that the bandwidth condition plays an important role 
in deciding the constant factor in the ramsey number of sparse graphs.  
This can be seen from a construction of Graham, R\"odl, and Ruci\'nski \cite{GrRoRu00} 
establishing that there exists a constant $c$ such that for every large enough $n$, there are $n$-vertex bipartite graphs with maximum degree at most $\Delta$ having Ramsey number at least $c^\Delta n$.

\medskip

The rest of the paper is organized as follows.
We begin by providing a brief outline of the proof of Theorem \ref{thm:main} in Section \ref{sec:outline}.
Then in Section \ref{sec:bipartite} we prove Theorem~\ref{thm:main_bipartite}, and
in Section \ref{sec:general} prove Theorem~\ref{thm:main}.
Using the tools developed in Section~\ref{sec:general}, we then proceed to Section~\ref{sec:application}
where we prove Theorems \ref{thm:main_bandwidth} and \ref{thm:ramsey_degenerate}.
We conclude the paper in Section \ref{sec:remarks} with some remarks.

\medskip

\noindent \textbf{Notation}. 
For an integer $m$, define $[m] := \{1,2,\ldots, m\}$ and for a pair of integers $m_1, m_2$,
define $[m_1, m_2] := \{m_1, \cdots, m_2\}$, $[m_1, m_2) := \{m_1, \cdots, m_2-1\}$,
$(m_1, m_2] := \{m_1 + 1, \cdots, m_2\}$, and $(m_1, m_2) := \{m_1+1, \cdots, m_2-1\}$.
For a set of elements $X$, we define $X^t = X \times \cdots \times X$ as the set of all $t$-tuples in $X$.
Throughout the paper, we will be using subscripts such as in $x_{\ref{thm:main}}$ to indicate 
that $x$ is the constant coming from Theorem/Lemma/Proposition \ref{thm:main}.

Let $G=(V,E)$ be an $n$-vertex graph.
For a vertex $v$ and a set $T$, we define $\deg(x; T)$ as the number of neighbors of $x$ in $T$, 
and define $\deg(x) := \deg(x; V)$.
Define $N(T) := \{ x : \{x,t\} \in E,\,\forall t \in T \}$ as the set of common neighbors of vertices in $T$.
For two vertices $v,w$ and a set of vertices $T$, we define $\codeg(v,w; T)$
as the number of common neighbors of $v$ and $w$ in $T$, and define $\codeg(v,w) := \codeg(v,w; V)$. 
For two sets $X$ and $Y$, define $E(X,Y) = \{(x,y) \in X \times Y \,:\, \{x,y\} \in E\}$
and $e(X,Y) = |E(X,Y)|$. Furthermore, define $e(X) = \frac{1}{2}e(X,X)$ as the number of 
edges in $X$. Let $d(X,Y) = \frac{e(X,Y)}{|X||Y|}$ be the density of edges bewteen
$X$ and $Y$.
If $G$ is bipartite with parts $V_1 \cup V_2$, then the {\em relative degree}
of a vertex $v_1 \in V_1$ is defined as $\frac{d(v_1)}{|V_2|}$, and of a vertex $v_2 \in V_2$
is defined as $\frac{d(v_2)}{|V_1|}$.
Recall that the relative minimum-degree of $G$ is the minimum of the relative-degree over
all vertices in $G$.

For a graph $H$, a {\em partial embedding} of $H$ on $V' \subseteq V(H)$ is an embedding of
$H[V']$ into $G$. For $V'' \supseteq V'$ an {\it extension to $V''$} of a partial 
embedding $f$ on $V'$ is an embedding $g : H[V''] \rightarrow G$ such that $g|_{V'} = f$.
We often abuse notation and denote the extended map using the same notation $f$.

\section{Outline of the Proof} \label{sec:outline}

We start by describing a brief outline of the proof.
Let $H$ be a $d$-degenerate graph, and let $G$ be a graph to which we wish to embed $H$.
Our embedding strategy is based on an iterative usage of dependent random choice, where
each round of iteration is similar to the one used by Kostochka and Sudakov \cite{KoSu},
and by Fox and Sudakov \cite{FoSu09}.

Suppose that $H$ and $G$ are both bipartite graphs with bipartition $W_1 \cup W_2$ and
$V_1 \cup V_2$, respectively.
Kostochka and Sudakov's strategy can be summarized in the following steps:
\begin{itemize}
  \setlength{\itemsep}{1pt} \setlength{\parskip}{0pt}
  \setlength{\parsep}{0pt}
  \item[1.] Randomly select $s'$ elements in $V_1$ to find a set $T \subseteq V_1$ for which all $(s+d)$-tuples of vertices in $A_2 := N(T)$ have at least $|V(H)|$ common neighbors in $V_1$.
  \item[2.] Randomly select $s$ elements in $A_2$ to find a set $S \subseteq A_2$ of size $|S| \le s$ for which all $d$-tuples of vertices in $A_1 := N(S)$ have at least $|V(H)|$ common neighbors in $A_2$.
  \item[3.] Embed $H$ into $A_1 \cup A_2$.
\end{itemize}
One key observation is that each $d$-tuple $Q$ of vertices in $A_2$ have at least $|V(H)|$ common 
neighbors in $A_1$. To see this, observe that $Q \cup S$ has at least $|V(H)|$ common neighbors in $V_1$ by Step 1 since $|Q \cup S| \le s + d$. Since $N(Q) \cap A_1 = N(Q) \cap N(S) \cap V_1 = N(Q \cup S) \cap V_1$, 
it implies that $Q$ has at least $|V(H)|$ common neighbors in $A_1$.
Hence in Step 3, one can embed vertices of $H$ one at a time, following the $d$-degenerate ordering
of $H$. Fox and Sudakov's strategy deviates from this in Step 1, where instead of imposing 
all $(s+d)$-tuples of vertices in $A_2$ have enough common neighbors, they bound the number
of $(s+d)$-tuples of vertices in $A_2$ having insufficient common neighbors. 
This then implies that most sets $S \subseteq A_2$ of size $|S| \le s$ has the property that 
for all $Q \subseteq A_2$ of size $|Q| \le d$, the set $S \cup Q$ has enough common neighbors.
The main advantage of this approach is that then we can take $s' = s$;
such difference results in an improvement on the bound.
Thus in Step 2, we just need to choose such `typical' $S$. 
The bound on the number of vertices of $H$ comes from the bound on the sizes of $A_1, A_2$, and the number of common neighbors of $d$-tuples.

Now suppose that $H$ is a larger graph but has small bandwidth, say $\beta$.
By using the strategy above, we can embed the initial segment of $H$ into $A_1 \cup A_2$ but
will at some point run out of space. Say that we have enough space to embed $2\beta$ vertices
of $H$. Start by embedding the initial $\beta$ vertices of $H$ into $A_1 \cup A_2$.
The key idea is to halt the embedding for a moment, 
and find another pair of sets $B_1 \cup B_2$ as above with the following 
additional property:
\begin{quote}
	(*) All $d$-tuples of vertices in $A_1$ have at least $2\beta$ common neighbors in $A_2 \cap B_2$, and all $d$-tuples of vertices in $A_2$ have at least $2\beta$ common neighbors in $A_1 \cap B_1$.
\end{quote}
Once we find such a pair we will embed the next $\beta$ vertices into $A_1 \cap B_1$ and $A_2 \cap B_2$
by invoking this property.
At this point, we can remove the image of the first $\beta$ vertices of $H$ from the graph $G$.
This is because there are no edges between the initial segment of $\beta$ vertices of $H$
and the vertices with label greater than $2\beta$. Note that we are left with a pair
$B_1 \cup B_2$ with $\beta$ vertices of $H$ embedded into it.
Hence repeating this process will eventually result in an embedding of $H$.

Thus the main challenge is to find a pair of sets $B_1 \cup B_2$ with the property listed above.
This can be done by first following the modified Steps 1 and 2 so that we impose
the additional property that `most' $(s+d)$-tuples
in $A_2$ have many common neighbors in $A_1$, and vice versa. Then,
\begin{itemize}
  \setlength{\itemsep}{1pt} \setlength{\parskip}{0pt}
  \setlength{\parsep}{0pt}
  \item[4.] Randomly select $s$ elements in $A_1$ to find a `typical' set $T' \subseteq A_1$ for which `most' $(s+d)$-tuples of vertices in $B_2 := N(T')$ have at least $2\beta$ common neighbors in $V_1$.
  \item[5.] Randomly select $s$ elements in $A_2 \cap B_2$ to find a `typical' set $S' \subseteq A_2 \cap B_2$ of size $|S'| \le s$ for which all $d$-tuples of vertices in $B_1 := N(S')$ have at least $2\beta$ common neighbors in $A_2\cap B_2$.
\end{itemize}
Note that since $T'$ is a `typical' set, for all $d$-tuple $Q$ of vertices in $A_1$, the set
$Q \cup T'$ has at least $2\beta$ common neighbors in $A_2$, i.e., $Q$ has at least
$2\beta$ common neighbors in $A_2 \cap N(T') = A_2 \cap B_2$. Similarly since $S'$ is a `typical' set, 
all $d$-tuples in $A_2$ will have at least $2\beta$ common neighbors in $A_1 \cap B_1$, thus
establishing (*).

One very important technical detail needs to be addressed. In Step 5, for our strategy to succeed, 
it is important that $A_2 \cap B_2$ is considerably larger than $2\beta$ so that a $d$-tuple
of vertices having less than $2\beta$ common neighbors in $A_2 \cap B_2$ is an uncommon situation.
This additional constraint has a chain effect on our proof.
First in Step 4, we need to impose this additional condition of $A_2 \cap B_2$ being large when choosing $A_1$.
Second, the size of $A_2 \cap B_2$ depends on the number of edges between $A_1$ and $A_2$, and thus in
Steps 1 and 2 we must have chosen $A_1$ and $A_2$ so that $d(A_1, A_2)$ is large enough.
Third, to prepare for the next step, when choosing $B_1$ and $B_2$ we must also impose that $d(B_1, B_2)$ is large enough. In the end, it turns out that we need to keep track of certain copies
of $C_4$ (or $K_{2,2}$) across various sets, and this is where having $G$ to be an
$(\varepsilon,\delta)$-dense graph becomes necessary. Obtaining such estimates become
complicated because our technique (dependent random choice) is based on the first moment method, and we must encode all this information into a single equation.
For bipartite graphs, there is a way to slightly modify Steps 4, 5 and the embedding strategy
to avoid using the $(\varepsilon,\delta)$-regular condition on $G$ (we do not gain in terms of simplicity). 
The reason we gave a description as above is because it works for non-bipartite graphs $H$ as well. 
For $r$-partite graphs, instead of $C_4$
we will be counting copies of complete $r$-partite graphs with $2$ vertices in each part.

\medskip

Suppose that $H$ is a $d$-degenerate graph with bandwidth $\beta$.
Note that since the given graph $H$ is $d$-degenerate there exists an ordering in which
all vertices have at most $d$ neighbors that precede itself, and
since $H$ has bandwidth $\beta$ there exists an ordering in which
all adjacent pairs are at most $\beta$ apart from each other. 
In the strategy sketched above, we implicitly used the following lemma
which asserts that there exists an ordering of the vertices achieving both.
A labelling of an $m$-vertex graph $H$ by $[m]$ is {\it $d$-degenerate} if each vertex
has at most $d$ neighbors preceding itself, and {\it $\beta$-local} if all adjacent
pairs are at most $\beta$ apart.

\begin{lem} \label{lem:bandwidth_degenerate}
Let $H$ be a $d$-degenerate graph with bandwidth $\beta$. 
Then there exists a $5d$-degenerate $\beta \log_2(4\beta)$-local labelling of the vertices 
of $H$. 
\end{lem}
\begin{proof}
Let $H$ be an $m$-vertex $d$-degenerate graph with bandwidth $\beta$ and denote $V = V(H)$. 
First label the vertices by $[m]$ so that adjacent vertices are at most $\beta$ apart. Call this labelling $\sigma : V(H) \rightarrow [m]$. We will construct another labelling $\pi : V(H) \rightarrow [m]$ which
has the desired property, based on an iterative algorithm where for $t = 0, 1, \ldots, m-1$, the $t$-th step of the algorithm defines the pre-image $\pi^{-1}(m-t)$. Suppose that we are at the $t$-th step of the algorithm
and let $V_t = V(H) \setminus \pi^{-1}(\{m-t+1, m-t+2, \ldots, m\})$.
Choose the vertex $v_t \in V_t$ having 
at most $5d$ neighbors preceding itself in $V_t$ in the labelling $\sigma$, with maximum $\sigma(v_t)$
(note that such vertex exists since $H$ is $d$-degenerate). 
Define $\pi(v_t) = m-t$. 
The definition implies that in the ordering $\pi$, each vertex has at most $5d$ neighbors
preceding itself. We claim that
\[
	-\beta \le \sigma(v) - \pi(v) \le \beta \log_2 \beta
\]
for all $v \in V(H)$. The lemma follows from this claim since if $\{v,w\}$ is an edge, then
\begin{align*}
	\pi(v) - \pi(w) 
	&\le (\pi(v) - \sigma(v)) + (\sigma(v) - \sigma(w)) + (\sigma(w) - \pi(w)) \\
	&\le \beta + \beta \log_2\beta + (\sigma(v) - \sigma(w))
	\le \beta \log_2 (4\beta).
\end{align*}

Fix a value of $t$ and let $M_t$ be the set of at most $2\beta$ largest labelled vertices of $V_t$ (according to $\sigma$). Note that $\sigma(M_t) \subseteq (m-t-2\beta, m]$.
Since $H$ is $d$-degenerate, it follows that $M_t$ spans at most $|M_t|d$ edges.
Also, by the bandwidth condition, the (at most) $\beta$ largest labelled vertices in $M_t$ have all its
neighbors in $V_t$ inside $M_t$. Therefore there must exist a vertex among them
having at most $\frac{|M_t|d}{\beta} \le 2d$ neighbors preceding itself in $\sigma$ in $V_t$. This implies that $\sigma(v_t) \ge m-t-\beta = \pi(v_t) - \beta$,
or equivalently, $\sigma(v_t) - \pi(v_t) \ge -\beta$.

Suppose that $\sigma(v_t) - \pi(v_t) > \beta \log_2 \beta$ for some $t$. 
Note that there exists $t' < t$ such that $\sigma(v_{t'}) \le m - t$ as otherwise 
$\{\sigma(v_1), \cdots, \sigma(v_{t-1})\} = (m-t, m]$ and
$\sigma(v_t) \le m-t = \pi(v_t)$.
Observe that  $\sigma(v_{t'}) \le m-t = \pi(v_t) \le \sigma(v_t) - \beta \log_2 \beta$.
Note that each $w \in V_t$ with $\sigma(w) > \sigma(v_t)$ has at least $5d$ neighbors
preceding itself in $\sigma$ in $V_t$ since otherwise we would have chosen $w$ instead 
of $v_t$. 
Define $I_j = (\sigma(v_{t})-j\beta, \sigma(v_{t})-(j-1)\beta]$ for $j \ge 0$, and $N_0 = \{v_t\}$.
Assume that a set $N_i$ which is a subset of the $i$-th neighborhood of $v_t$, 
satisfying $|N_i| \ge 2^i$ and $N_i \subseteq I_{j_i}$ for some $j_i \le i$ is given.
Define $N_{i+1}' = \{v \in V_{t'} : \exists w \in N_i \cap N(v),\,\sigma(v) < \sigma(w) \}$.
Note that $N_{i+1}' \cup N_i$ spans less than $(|N_{i+1}'| + |N_i|)d$ edges by the degeneracy condition.
Furthermore if $i \le \log \beta$, then all vertices in $N_i$
succeed $v_{t'}$ in $\sigma$ and thus have at least $5d$ neighbors preceding itself in $\sigma$. 
Therefore $N_{i+1}' \cup N_i$ spans at least $|N_i|\cdot 5d$ edges. Hence
\[
	|N_i| \cdot 5d < (|N_{i+1}'| + |N_i|)d,
\]
and we have $|N_{i+1}'| > 4|N_i|$. Since $N_i \subseteq I_{j_i}$, we have $N_{i+1}' \subseteq I_{j_i} \cup I_{j_i + 1}$. Therefore either $N_{i+1}' \cap I_{j_i}$ or $N_{i+1}' \cap I_{j_i+1}$ has size greater than $2|N_i|$.
Define $N_{i+1}$ as the intersection having size greater than $2|N_i| \ge 2^{i+1}$.
This is a contradiction for $i = \log \beta$ since each interval has size at most $\beta$. 
Hence the claim $\sigma(v_t) - \pi(v_t) \le \beta \log_2 \beta$ holds.
\end{proof}

\section{Bipartite graphs} \label{sec:bipartite}

Throughout this section, fix a bipartite graph $H$.
We will consider the problem of embedding $H$ into another bipartite graph $G$.
We assume that the bipartition $W_1 \cup W_2$ of $H$ and $V_1 \cup V_2$ of $G$ 
are given. We restrict our attention
to finding an embedding $f : V(H) \rightarrow V(G)$ for which $f(W_1) \subseteq V_1$ and
$f(W_2) \subseteq V_2$ even when we do not explicitly refer to the bipartition of $H$.

A set $X$ is {\it $(d,\beta)$-common into $Y$} if every $d$-tuple of vertices of $X$
has at least $\beta$ common neighbors in $Y$.
A pair of sets $(X,Y)$ is {\it $(d,\beta)$-common} if $X$ is $(d,\beta)$-common into $Y$
and $Y$ is $(d,\beta)$-common into $X$.
Let $G$ be a bipartite graph with partition $V_1 \cup V_2$. 
The usefulness of $(d,\beta)$-common property has been noticed over years
and researchers developed powerful methods such as
dependent random choice to find subsets that are $(d,\beta)$-common
(see, e.g., the survey paper of Fox and Sudakov \cite{FoSu11}). 
The following lemma shows how we will use the $(d,\beta)$-common property.

\begin{lem} \label{lem:extend_partial}
Let $G$ be a bipartite graph with bipartition $V_1 \cup V_2$, and let $H$ be a 
bipartite graph with a $d$-degenerate $\beta$-local vertex labelling by $[\beta'+2\beta]$.
Suppose that pairs of sets $(A_1, A_2)$ and $(B_1, B_2)$ with
$A_1, B_1 \subseteq V_1$ and $A_2, B_2 \subseteq V_2$ satisfy the following property:
\begin{itemize}
  \setlength{\itemsep}{1pt} \setlength{\parskip}{0pt}
  \setlength{\parsep}{0pt}
  \item[(i)] $A_1$ is $(d,\beta'+2\beta)$-common into $A_2$,
  \item[(ii)] $A_2$ is $(d,\beta'+2\beta)$-common into $A_1 \cap B_1$, and
  \item[(iii)] $A_1 \cap B_1$ is $(d,\beta'+2\beta)$-common into $A_2 \cap B_2$.
\end{itemize}
Then every partial embedding  $f : [\beta'] \rightarrow V(G)$ can be extended
into an embedding $f : [\beta'+2\beta] \rightarrow V(G)$ where all vertices with label
larger than $\beta'$ are mapped into $B_1 \cup B_2$.
\end{lem}
\begin{proof}
We will extend the partial embedding one vertex at a time according to the order given 
by the vertex labelling.
Suppose that the partial embedding has been defined over $[t-1]$ for some $t > \beta'$, 
and we are about to embed vertex $t$. 
There are at most $d$ neighbors of $t$ that precede $t$. 
Call these vertices $v_1, v_2, \ldots, v_{d'}$ for some $d' \le d$, and let $w_i = f(v_i)$ be their images.
First suppose that $\{w_i\}_{i=1}^{d'} \subseteq A_2$. Then
by the given condition, since $d' \le d$, the $d'$-tuple $(w_1, w_2, \ldots, w_{d'})$
has at least $\beta'+2\beta$ common neighbors in $A_1 \cap B_1$. 
Since $H$ consists of at most $\beta'+2\beta$ vertices, 
there exists at least one vertex $x \in A_1 \cap B_1$ that is adjacent to 
all vertices $w_1, w_2, \ldots, w_{d'}$
and is not an image of $f$. By defining $f(t) = x$, we can extend $f$ to $[t]$.
Now suppose that $\{w_i\}_{i=1}^{d'} \subseteq A_1$.
If $d \le \beta'+\beta$, then the same strategy as above allows us to embed $f(t)$ into $B_1$.
If $d > \beta'+\beta$, then note that by the bandwidth condition
$\{v_i\}_{i=1}^{d'} \subseteq (\beta',\beta'+2\beta]$ and thus
$\{w_i\}_{i=1}^{d'} \subseteq A_1 \cap B_1$. Then the argument above shows that
we can embed $f(t)$ into $A_2 \cap B_2$.
In the end, we obtain the desired extension of $f$.
\end{proof}

Suppose that a partial embedding $f : [\beta] \rightarrow A_1 \cup A_2$ is given.
Our proof proceeds by finding a pair of subsets $(B_1, B_2)$ and applying
Lemma \ref{lem:extend_partial} to extend $f$ 
so that the next segment embeds into $B_1 \cup B_2$. 
Afterwards, we repeat the same process to further extend $f$ until
we find an embedding of the whole graph $H$.
Thus the key step of the proof is to find such $(B_1, B_2)$ given $(A_1, A_2)$.

We introduce one important concept before stating the lemma that serves this purpose.
Given a pair of sets $X$ and $Y$ in a graph $G$,
define the {\it $(p,d,\beta)$-potential of $X$ in $Y$} as the number of 
$(p+d)$-tuples in $X$ that have fewer than $\beta$ common neighbors in $Y$.
We say that $X$ has {\it $\lambda$-negligible $(p,d,\beta)$-potential in $Y$}
if its $(p,d,\beta)$-potential in $Y$ is less than $\lambda^{p-1}$.
We may simply denote the $(p,d,\beta)$-potential as $p$-potential when
$d$ and $\beta$ are clear from the context.
Note that if $X$ has $\lambda$-negligible $0$-potential, then the number of $d$-tuples
in $X$ that have fewer than $\beta$ common neighbors is less than $\lambda^{-1} < 1$,
i.e., $X$ is $(d,\beta)$-common into $Y$.
Conversely for all $p \ge 0$, if $X$ is $(p+d,\beta)$-common into $Y$, then
$X$ has $\lambda$-negligible $(p,d,\beta)$-potential in $Y$.
In fact, the reason we have introduced $p$-potential is because it is `easier'
to impose small $(p,d,\beta)$-potential than to impose $(p+d,\beta)$-commonness.
In the bipartite setting, introducing this additional concept results in a better bound
and the theorem can be proved even without the concept
(this is how the bound on the Ramsey number of degenerate graphs of Fox and Sudakov \cite{FoSu09} 
differs from that of Kostochka and Sudakov \cite{KoSu}).
However for non-bipartite graphs we heavily rely on this concept and thus we
introduce it here to prepare the reader for the more technical part that will later come
(and for the better bound).

\begin{lem} \label{lem:rolling_drc_degenerate}
For natural numbers $n, d$ and positive real numbers $\varepsilon, \gamma, \delta$, define
$s = \sqrt{\frac{d \log n}{\log(2/\delta)}}$ and suppose that 
$\lambda \le \left(\frac{\delta}{2}\right)^{5s}n$, 
$\beta \le \left(\frac{\lambda}{n}\right)^3 \gamma n$, and $\varepsilon < \frac{\delta}{8}$.
For sufficiently large $n$, 
let $G$ be a bipartite graph on at most $n$ vertices with parts $V_1 \cup V_2$ satisfying
$|V_1|, |V_2| \ge \gamma n$, where at least $(1-\varepsilon)|V_1|$ vertices in $V_1$ 
have degree at least $\delta|V_2|$ in $V_2$, and vice versa for vertices in $V_2$.
Suppose that sets $A_1 \subseteq V_1$ and $A_2 \subseteq V_2$ for which
$A_2$ contains at least $\frac{1}{4}\left(\frac{\delta}{2}\right)^{2s}|V_2|$ vertices of degree at least $\frac{\delta}{2}|V_1|$ 
is given.
Then there exist sets $B_1 \subseteq V_1$ and $B_2 \subseteq V_2$
satisfying the following properties:
\begin{itemize}
  \setlength{\itemsep}{1pt} \setlength{\parskip}{0pt}
  \setlength{\parsep}{0pt}
  \item[(i)] $B_1$ is $(d,2\beta)$-common into $A_2 \cap B_2$,
  \item[(ii)] $B_2$ has $\lambda$-negligible $(s,d,2\beta)$-potential in $B_1$, and
  \item[(iii)] $B_2$ contains at least $\frac{1}{2}\left(\frac{\delta}{2}\right)^{2s}|V_2|$ vertices of degree at least
  $\delta|V_1|$.
\end{itemize}
Furthermore, if $A_2$ has $\lambda$-negligible $(s,d,\beta)$-potential in $A_1$, then
\begin{itemize}
  \item[(iv)] $A_2$ is $(d,\beta)$-common into $A_1 \cap B_1$,  
\end{itemize}
\end{lem}

The proof of Lemma~\ref{lem:rolling_drc_degenerate} will be given in the next subsection.
We first prove the main result of this section using Lemmas \ref{lem:extend_partial} and \ref{lem:rolling_drc_degenerate}.
A bipartite graph $G$ with bipartition $(V_1, V_2)$ is {\it $(\alpha, \varepsilon, \delta)$-degree-dense}
if for every pair of subsets $(X_1, X_2)$ satisfying $X_i \subseteq V_i$ and $|X_i| \ge \alpha |V_i|$
(for $i=1,2$), the subgraph $G$ induced on $X_1 \cup X_2$ contains at least
$(1-\varepsilon)|X_i|$ vertices in $X_i$ of relative degree at least $\delta$ (for $i=1,2$).
Note that if $G$ is $(\varepsilon,\delta)$-dense, then it is
$(\varepsilon_1, \varepsilon_2, \delta)$-degree-dense for all $\varepsilon_1 \varepsilon_2 \ge \varepsilon$.
Also, if $G$ has relative minimum degree $\alpha$, then it is
$(1-\alpha+\delta, 0, \delta)$-degree-dense for all $\delta < \alpha$.
Hence Theorem \ref {thm:main_bipartite} and the bipartite case of Theorem~\ref{thm:main} 
follows from the following theorem. As mentioned before, the techniques used in this section
is tailored to bipartite graphs. As a result, the bound on the bandwidth that we obtain here
for the bipartite case of Theorem~\ref{thm:main} is in fact better than the bound that we obtain
through the proof of the general case of Theorem~\ref{thm:main} that will be given in a later section.

\begin{thm} \label{thm:bipartite_general}
Let $\varepsilon, \delta, \alpha, \gamma$ be fixed real numbers 
satisfying $0 \le \varepsilon < \frac{\delta}{8}$.
Let $H$ be a $d_0$-degenerate bipartite graph with parts $W_1 \cup W_2$ having 
bandwidth $\beta_0 \le e^{-50\sqrt{\log(1/\delta) d_0 \log n}} \gamma n$.
If $n$ is sufficiently large and 
$G$ is an $n$-vertex $(\alpha, \varepsilon, \delta)$-degree-dense bipartite graph with bipartition 
$V_1 \cup V_2$ satisfying $|V_i| \ge \max\{\gamma n, \frac{1}{1-\alpha}|W_i|\}$ for $i=1,2$, then
$G$ contains a copy of $H$.
\end{thm}
\begin{proof}
Let $H$ be a given $m$-vertex $d_0$-degenerate bipartite graph of bandwidth $\beta_0$
with bipartition $W_1 \cup W_2$.
For $d = 5d_0$ and $\beta = \beta_0 \log_2 (4\beta_0)$,
by Lemma~\ref{lem:bandwidth_degenerate} there exists a 
$d$-degenerate $\beta$-local labelling of the vertex set of $H$.
Define $I_i = (2(i-1)\beta, 2i\beta] \cap [m]$ and $J_i = [2i\beta] \cap [m]$ for $i \ge 0$.
For $i < 0$, define $I_i = J_i = \emptyset$.

Define $s = \sqrt{\frac{d \log n}{\log(2/\delta)}}$ and $\lambda = (\frac{\delta}{2})^{5s} n$.
Note that $\beta \le \frac{1}{4}(\frac{\lambda}{n})^{3} \alpha \gamma n$.
Let $G$ be a given $n$-vertex bipartite graph with bipartition $V_1 \cup V_2$, where
$|W_1| \le (1-\alpha)|V_1|$ and $|W_2| \le (1-\alpha)|V_2|$.
Let $V = V(G)$ and $n = |V|$. Our goal is to find an embedding $f$ of $H$
to $G$ for which $f(W_i) \subseteq V_i$ for $i=1,2$.
We will embed $H$ using an iterative algorithm. 
For $t \ge 1$, at the beginning of the $t$-th step we will be given a pair
$(A_1, A_2)$ with $A_1 \subseteq V_1$ and $A_2 \subseteq V_2$ and a partial embedding
$f : J_{t-1} \rightarrow V(G)$ satisfying the following properties:
\begin{itemize}
  \setlength{\itemsep}{1pt} \setlength{\parskip}{0pt}
  \setlength{\parsep}{0pt}
  \item[(i)] $f(I_{t-1}) \subseteq A_1 \cup A_2 \subseteq V \setminus f(J_{t-2})$,
  \item[(ii)] $A_1$ is $(d,4\beta)$-common into $A_2$,
  \item[(iii)] $A_2$ has $\lambda$-negligible $(s,d,4\beta)$-potential in $A_1$,
  \item[(iv)] $A_2$ contains at least $\frac{1}{4}\left(\frac{\delta}{2}\right)^{2s}|V_2 \setminus f(J_{t-2})|$ vertices of 
  relative-degree at least $\frac{\delta}{2}$ in $V_1 \setminus f(J_{t-2})$.
\end{itemize}
Then at the $t$-th step we will construct sets $C_1 \subseteq V_1$ and
$C_2 \subseteq V_2$, and extend $f$ to $I_t$ so that 
the conditions above for the $(t+1)$-th step is satisfied.
We can then continue the process until we finish embedding $H$.

Initially, apply Lemma \ref{lem:rolling_drc_degenerate} to the sets
$(A_1)_{\ref{lem:rolling_drc_degenerate}} = V_1$,
$(A_2)_{\ref{lem:rolling_drc_degenerate}} = V_2$, 
and $\beta_{\ref{lem:rolling_drc_degenerate}} = 2\beta$ to obtain sets
$A_1^{(0)}$ and $A_2^{(0)}$ satisfying Properties (ii) and (iii).
Note that Properties (i) and (iv) are vacuously true since $I_i = J_i = \emptyset$ for $i < 0$.
Suppose that for some $t \ge 1$, we completed the $(t-1)$-th step of the algorithm.
Let $G_t$ be the subgraph of $G$ induced on $V \setminus f(J_{t-2})$.
Since $G$ is $(\alpha,\varepsilon, \delta)$-degree-dense and 
$|W_i| \le (1-\alpha)|V_i|$
for $i=1,2$, it follows that at least $1-\varepsilon$ proportion
of vertices of each part of the bipartition of $G_t$
has relative degree at least $\delta$.
Note that $|V_i \cap V(G_t)| \ge \alpha|V_i| \ge \alpha \gamma n$ for $i=1,2$. 
By Properties (iii) and (iv) we may apply Lemma~\ref{lem:rolling_drc_degenerate} 
to the pair $(A_{1}, A_{2})$
in the graph $G_t$ with $\beta_{\ref{lem:rolling_drc_degenerate}} = 4\beta$ 
and $\gamma_{\ref{lem:rolling_drc_degenerate}} = \alpha \gamma$ 
to obtain sets $B_1$ and $B_2$ satisfying the following properties:
\begin{itemize}
  \setlength{\itemsep}{1pt} \setlength{\parskip}{0pt}
  \setlength{\parsep}{0pt}
  \item[(a)] $B_1$ is $(d,8\beta)$-common into $A_2 \cap B_2$,
  \item[(b)] $B_2$ has $\lambda$-negligible $(s,d,8\beta)$-potential in $B_1$,
  \item[(c)] $B_2$ contains at least $\frac{1}{2}\left(\frac{\delta}{2}\right)^{2s}|V_2 \setminus f(J_{t-2})|$ vertices of degree at least $\delta|V_1 \setminus f(J_{t-2})|$, and
  \item[(d)] $A_2$ is $(d,4\beta)$-common into $A_1 \cap B_1$,  
\end{itemize}
Note that $f(I_{t-1}) \subseteq V(G_t)$. Let $g = f|_{I_{t-1}}$.
Since $g$ is an embedding of $H[I_{t-1}]$ into $G_t$, it can also be viewed
as a partial embedding of $H[I_{t-1} \cup I_t]$ into $G_t$ defined on $I_{t-1}$.
Thus by Property (ii) and Properties (a), (d),
we can apply Lemma~\ref{lem:extend_partial} and extend $g$
to $I_t$ so that $g(I_t) \subseteq (A_{1} \cap B_1) \cup (A_2 \cap B_2) \subseteq B_1 \cup B_2$.
Furthermore, since the labelling is $\beta$-local, there are no edges of $H$
between $I_t$ and $J_{t-2}$. Thus $g$ in fact extends $f$ to $I_{t}$.

Define $C_1 = B_1 \setminus f(I_{t-1})$ and $C_2 = B_2 \setminus f(I_{t-1})$.
The sets $f(I_{t-1})$ and $f(I_t)$ are disjoint since $f$ is a partial embedding.
Thus $f(I_t) \subseteq C_1 \cup C_2$, and Property (i) for the next step is satisfied.
Define $G_{t+1}$ as the subgraph of $G$ induced on $V \setminus f(J_{t-1})$.
Note that $G_{t+1}$ is obtained from $G_t$ by removing $2\beta$ vertices.
Property (ii) for the next step then follows from Property (a).
By Property (b), $B_2$ has
$\lambda$-negligible $(s,d,8\beta)$-potential in $B_1$ in the graph $G_t$.
An $(s+d)$-tuple in $B_2$ can have less than $4\beta$ common neighbors in $C_1$ 
only if it has less than $6\beta$ common neighbors in $B_1$.
Thus $B_2$ has $\lambda$-negligible $(s,d,6\beta)$-potential in $C_1$, and since
$C_2 \subseteq B_2$, it follows that $C_2$ has $\lambda$-negligible 
$(s,d,6\beta)$-potential in $C_1$, proving Property (iii) for the next step.
Furthermore by Property (c),
there are at least $\frac{1}{2}\left(\frac{\delta}{2}\right)^{2s}|V_2 \setminus f(J_{t-2})|$ vertices 
in $B_2$ of degree at least $\delta |V_1 \setminus f(J_{t-2})|$ in the graph $G_t$. Therefore
$C_2$ contains at least $\frac{1}{2}\left(\frac{\delta}{2}\right)^{2s}|V_2 \setminus f(J_{t-2})|- 2\beta \ge \frac{1}{4}\left(\frac{\delta}{2}\right)^{2s}|V_2 \setminus f(J_{t-1})|$ vertices 
of degree at least 
$\delta |V_1 \setminus f(J_{t-2})| - 2\beta \ge \frac{\delta}{2}|V_1 \setminus f(J_{t-1})|$ in the graph $G_{t+1}$. 
This proves Property (iv) for the next step and concludes the proof.
\end{proof}

Theorem~\ref{thm:main_bipartite} straightforwardly follows.

\begin{proof}[Proof of Theorem~\ref{thm:main_bipartite}]
Let $G$ be an $n$-vertex graph of minimum degree at least $\gamma n$, and let $V = V(G)$.
Let $H$ be a $d$-degenerate bipartite graph with parts $W_1 \cup W_2$ having bandwidth at most
$e^{-100\sqrt{d\log(1/\varepsilon) \log n}} n$ where $|W_1| + |W_2| \le (\gamma - \varepsilon)n$.
Define $w_1 = |W_1| + \frac{\varepsilon}{4}n$ and
$w_2 = |W_2| + \frac{\varepsilon}{4}n$.

Let $V_1 \cup V_2$ be a bipartition of $V$ with $|V_1| = \frac{w_1}{w_1 + w_2} n$
and $|V_2| = \frac{w_2}{w_1 + w_2} n$ chosen uniformly at random.
Standard esimates on concentration of hypergeometric inequality shows that
the bipartite subgraph of $G$ induced on $V_1 \cup V_2$ has relative minimum 
degree at least $\gamma - \frac{\varepsilon}{4}$ with probability $1-o(1)$.
Let $V_1 \cup V_2$ be a particular partition where such bound holds and 
let $G'$ be the bipartite subgraph induced on $V_1 \cup V_2$.
Since $G'$ has relative minimum degree at least $\gamma - \frac{\varepsilon}{4}$, 
it is $(1 - \gamma + \frac{\varepsilon}{2}, 0, \frac{\varepsilon}{4})$-degree-dense.
Note that $\min\{|V_1|, |V_2|\} \ge \frac{\varepsilon}{4}n$. Furthermore,
\[
	|V_1| = \frac{w_1}{w_1 + w_2} n 
	\ge \frac{|W_1| + (\varepsilon/4)n}{\gamma - (\varepsilon/2)}
	\ge \frac{|W_1|}{\gamma - (\varepsilon/2)},
\]
and similar bound holds for $V_2$.
Therefore we can apply Theorem~\ref{thm:bipartite_general} with $(\varepsilon)_{\ref{thm:bipartite_general}} = 0$, 
$(\delta)_{\ref{thm:bipartite_general} } = \frac{\varepsilon}{4}$,
$(\alpha)_{\ref{thm:bipartite_general} } = 1 - \gamma + \frac{\varepsilon}{2}$, and
$(\gamma)_{\ref{thm:bipartite_general}} = \frac{\varepsilon}{4}$ to find a copy of $H$
in $G$.
\end{proof}

\subsection{Proof of Lemma \ref{lem:rolling_drc_degenerate}}

In this subsection we prove Lemma \ref{lem:rolling_drc_degenerate} using dependent random choice.
We will take $B_2 = N(T_1)$ for some appropriately chosen random set $T_1 \subseteq V_1$.
The main challenge is that the events that we would like to control are not concentrated (more precisely,
we do not know how to prove that they are concentrated), and hence we need to encode
all the events into a single random variable and use the first moment method.
For example if we want to show that $|B_2| \ge b$, then we can show
\[
	\BBE[|B_2| - b] \ge 0.
\] 
Then there exists a choice of $T_1$
for which $|B_2| \ge b$. Our situation is more complicated. We want to find $B_2$ so that
$|B_2| \ge b$ and $|B_2 \cap A_2| \ge a$ simultaneously holds (among other properties). 
We can try to prove that $\BBE[|B_2||B_2 \cap A_2| - a b] \ge 0$, but then we have no bound on individual sets.
This is critical for us since we want a lower bound on $B_2$ that is independent of the size of $A_2$ (otherwise the set $B_2$ will rapidly shrink between iterations and our embedding strategy will fail). 
What we can do instead is prove that
\[
	\BBE[|B_2||B_2 \cap A_2| - a |B_2| - b|A_2 \cap B_2|] \ge 0.
\]
As long as $a$ and $b$ are positive real numbers, this will imply the desired bounds.
In fact we do not need an individualized bound on $|A_2 \cap B_2|$ and hence our random variable
will be of the form
\[
	\BBE[|B_2||B_2 \cap A_2| - a |B_2| - ab] \ge 0,
\]
which is easier to verify. We also need bounds on certain potentials. 
For two sets $X$ and $Y$, denote the $(p,d,\beta)$-potential of $X$ in $Y$ as $\xi_{p,d,\beta}(X,Y)$.
If we want to prove that the $(p,d,\beta)$-potential of $B_2$ in $V_1$
is $\lambda$-negligible, i.e., $\xi_{p,d,\beta}(B_2, V_1) < \lambda^{p-1}$, then
we will modify the equation into
\[
	\BBE\left[|B_2||B_2 \cap A_2| - a |B_2| - ab - \frac{n^2\xi_{p,d,\beta}(B_2, V_1)}{\lambda^{p-1}}\right] > 0,
\]
as it will force $\xi_{p,d,\beta}(B_2, V_1) < \frac{|B_2||B_2 \cap A_2|}{n^2} \lambda^{p-1} \le \lambda^{p-1}$.
In some cases we can only bound the potential in terms of another potential
and need control on the ratio between two potentials.
The following proposition captures the strength and beauty of dependent random choice.

\begin{prop} \label{prop:negligible_potential}
Let $s, d, p$ be integers.
Let $X$ and $Y$ be a given pair of subsets of vertices and let
$X' \subseteq X$ be a subset of size at least $m$.
Let $\hat{T}$ be a $s$-tuple of vertices in $X'$ chosen
independently and uniformly at random. Then the following hold.
\begin{itemize}
  \setlength{\itemsep}{1pt} \setlength{\parskip}{0pt}
  \setlength{\parsep}{0pt}
\item[(i)] 
\[
	\BBE\left[\frac{\xi_{p, d, \beta}(N(\hat{T}) \cap Y, X)}{ \xi_{p, d, \beta}(Y, X)} \right] 
	\le \left( \frac{\beta}{m} \right)^{s}.
\]
\item[(ii)] 
\[
	\BBE\left[\frac{\xi_{p, d, \beta}(X, N(\hat{T}) \cap Y)}{\xi_{p+s, d, \beta}(X, Y)} \right] \le \left(\frac{1}{m} \right)^{s}.
\]
\end{itemize}
\end{prop}
\begin{proof}
(i) Let $Q$ be a fixed $(p+d)$-tuple in $Y$ with less than $\beta$ common neighbors in $X$.
The probability that $Q \subseteq N(\hat{T})$ is 
\[
	\left(\frac{|N(Q) \cap X'|}{|X'|}\right)^{s}
	\le \left(\frac{\beta}{m}\right)^{s}.
\]
Therefore
\[
	\BBE\left[\xi_{p, d, \beta}(N(\hat{T}) \cap Y, X)\right]
	\le \left(\frac{\beta}{m}\right)^{s} \xi_{p, d, \beta}(Y, X),
\]
and Part (i) follows.

\medskip

\noindent
(ii) Note that a $(p+d)$-tuple $Q$ in $X$ has less than $\beta$ common neighbors in $N(\hat{T}) \cap Y$ 
if and only if the $(p+d+s)$-tuple $Q \cup \hat{T}$ has less than $\beta$ common neighbors in $Y$.
Therefore,
\[
	\BBE[\xi_{p, d, \beta}(X, N(\hat{T}) \cap Y)]
	= q \xi_{p+s, d, \beta}(X;Y),
\]
where $q$ is the probability that $\hat{T}$ is precisely the last $s$ elements of $Q'$ for a fixed $(p+d+s)$-tuple $Q'$ in $X$.
Part (ii) follows since $q = \frac{1}{m^s}$.
\end{proof}

We now prove Lemma \ref{lem:rolling_drc_degenerate}.
Let $G$ be a bipartite graph with bipartition $V_1 \cup V_2$ where at least $1-\varepsilon$
proportion of vertices of each side have relative-degree at least $\delta$. 
Suppose that sets $A_1 \subseteq V_1$ and $A_2' \subseteq A_2 \subseteq V_2$ are given where
$|A_2'| \ge \frac{1}{4}\left(\frac{\delta}{2}\right)^{2s}|V_2|$ and all vertices in $A_2'$ have degree
at least $\frac{\delta}{2}|V_1|$ in $V_1$. 
Let $V_1' \subseteq V_1$ and $V_2' \subseteq V_2$ be the set of vertices of 
relative-degree at least $\delta$. The given condition implies that
\[
	|V_1'| \ge (1-\varepsilon)|V_1|
	\quad \text{and} \quad
	|V_2'| \ge (1-\varepsilon)|V_2|.
\]
Throughout the proof we will be considering $(p,d,\beta)$-potentials for fixed $d$
and two different values of $\beta$. 
Hence to avoid having three subscripts such as in $\xi_{p,d,\beta}$, 
we abuse notation and for sets $X$ and $Y$ and an integer $p$, 
we denote the the $(p,d,\beta)$-potential of $X$ in $Y$ as $\xi_{p}(X,Y)$
and the $(p,d,2\beta)$-potential of $X$ in $Y$ as $\eta_{p}(X, Y)$.

Let $\hat{T}_1$ be a random multi-set of $s$ vertices in $V_1'$, where each vertex is chosen 
uniformly and independently at random. Set $\hat{B}_2 = N(\hat{T}_1)$ and
$\hat{B}'_2 = N(\hat{T}_1) \cap V_2'$.
Define
\[
	\mu := \BBE\left[ |\hat{B}_2 \cap A_2'| \cdot |\hat{B}_2'| - \left(\frac{\delta}{2}\right)^{2s} |V_2'| \cdot |\hat{B}_2 \cap A_2'| - \frac{n^{2} }{\lambda^{2s-1}}\eta_{2s}(\hat{B}_2, V_1) \right],
\]
and let $T_1$ be a particular choice of $\hat{T}_1$ for which the random variable on the right-hand side 
becomes at least its expected value. 
Let $B_2 = \hat{B}_2$ and $B_2' = \hat{B}_2'$ for this choice of $\hat{T}_1$.

Let $\hat{T}_2$ be a random multi-set of $s$ vertices in $A_2' \cap B_2$, where each vertex is chosen 
uniformly and independently at random. Set $\hat{B}_1 = N(\hat{T}_2)$ and define
\[
	\nu := \BBE\left[ \eta_{0}(\hat{B}_1, A_2 \cap B_2) + \frac{\lambda^{s}\eta_{s}(B_2, \hat{B}_1)}{\eta_{2s}(B_2, V_1)} + \frac{\lambda^{s}\xi_{0}(A_2, \hat{B}_1 \cap A_1)}{\xi_{s}(A_2, A_1)} \right],
\]
and let $T_2$ be a particular choice of $\hat{T}_2$ for which the random variable on the right-hand side 
becomes at most its expected value.  Let $B_1 = \hat{B}_1$ for this choice of $\hat{T}_2$
(see Figure \ref{fig:fig_bipartite}).

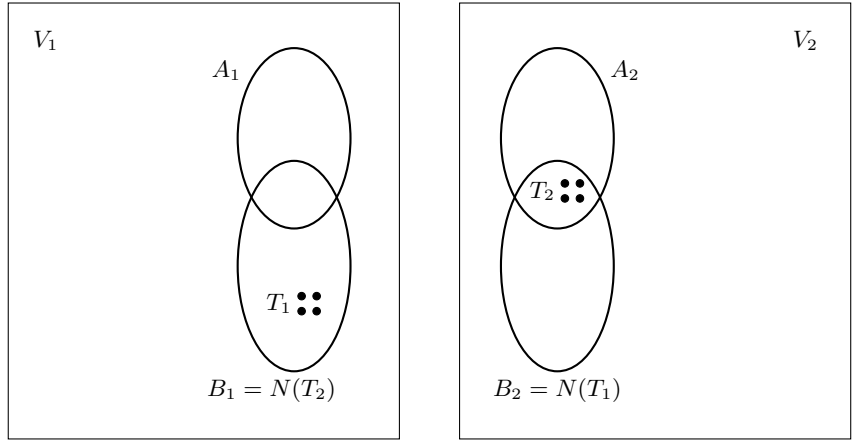
\begin{figure}[h]
  \centering
  \begin{tabular}{ccc}

\begin{tikzpicture}

  \draw (-2.8,-2.5) rectangle (2.4,3.3);
  \draw (-2.3,2.8) node {\footnotesize{$V_1$}};

  \draw [thick] (1,1.5) ellipse (0.75cm and 1.2cm);
  \draw (0.1,2.4) node {\footnotesize{$A_1$}};

  \draw [thick] (1,-0.2) ellipse (0.75cm and 1.4cm);
  \draw (0.7,-1.85) node {\footnotesize{$B_1= N(T_2)$}};

  \draw [fill=black] (1.1,-0.6) circle (0.5mm);
  \draw [fill=black] (1.1,-0.8) circle (0.5mm);
  \draw [fill=black] (1.3,-0.6) circle (0.5mm);
  \draw [fill=black] (1.3,-0.8) circle (0.5mm);
  \draw (0.8,-0.7) node {\footnotesize{$T_1$}};

  \draw (3.2,-2.5) rectangle (8.4,3.3);
  \draw (7.8,2.8) node {\footnotesize{$V_2$}};

  \draw [thick] (4.5,1.5) ellipse (0.75cm and 1.2cm);
  \draw (5.4,2.4) node {\footnotesize{$A_2$}};

  \draw [thick] (4.5,-0.2) ellipse (0.75cm and 1.4cm);
  \draw (4.5,-1.85) node {\footnotesize{$B_2 = N(T_1)$}};

  \draw [fill=black] (4.6,0.9) circle (0.5mm);
  \draw [fill=black] (4.6,0.7) circle (0.5mm);
  \draw [fill=black] (4.8,0.9) circle (0.5mm);
  \draw [fill=black] (4.8,0.7) circle (0.5mm);
  \draw (4.3,0.8) node {\footnotesize{$T_2$}};

\end{tikzpicture}
  \end{tabular}
  \caption{The sets $V_i, A_i, B_i, T_i$ for $i=1,2$. We first choose $T_1$, and then $T_2$.}
 \label{fig:fig_bipartite}
\end{figure}

\begin{claim} \label{clm:claim0}
$\mu \ge \frac{1}{8} \left( \frac{\delta}{2}\right)^{4s} |V_2|^2$ and $\nu < 1$.
\end{claim}

We first deduce the lemma provided that the claim holds.
Property (iii) follows since
\[
	|B_2 \cap A_2'| |B_2'| - \left(\frac{\delta}{2}\right)^{2s} |V_2'| |B_2 \cap A_2'|
	\ge \mu > 0,
\]
and  $|V_2'| \ge (1-\varepsilon)|V_2|$
implies $|B_2'| \ge (\frac{\delta}{2})^{2s}|V_2'| \ge (1-\varepsilon)(\frac{\delta}{2})^{2s}|V_2|$.
Furthermore since
\begin{align} \label{eq:potential_previous_step}
	\frac{n^{2}}{\lambda^{2s-1}} \eta_{2s}(B_2, V_1)
	< |B_2 \cap A_2| \cdot |B_2'| \le n^2
	\quad \Longrightarrow \quad
	\eta_{2s}(B_2,V_1) < \lambda^{2s-1}.
\end{align}

Since $\nu < 1$, we see that
\[
	 \eta_{0}(B_1, A_2 \cap B_2) + \frac{\lambda^{s}\eta_{s}(B_2, B_1)}{\eta_{2s}(B_2, V_1)} + \frac{\lambda^{s}\xi_{0}(A_2, B_1 \cap A_1)}{\xi_{s}(A_2, A_1)} < 1.
\]
The first term on the left-hand side gives $\eta_{0}(B_1, A_2 \cap B_2) < 1$.
Thus $B_1$ is $(d,2\beta)$-common into $A_2 \cap B_2$ and Property (i) follows.
The second term gives $\eta_{s}(B_2, B_1) < \lambda^{-s}\eta_{2s}(B_2, V_1) < \lambda^{s-1}$ (from
\eqref{eq:potential_previous_step}), establishing
that $B_2$ has $\lambda$-negligible $(s,d,2\beta)$-potential into $B_1$, which is Property (ii).
The third term gives $\xi_{0}(A_2, B_1 \cap A_1) < \lambda^{-s}\xi_{s}(A_2, A_1)$. Hence
if $A_2$ has $\lambda$-negligible $(s,d,\beta)$-potential into $A_1$, then
$\xi_{s}(A_2, A_1) < \lambda^{s-1}$ and thus $\xi_0(A_2, B_1 \cap A_1) = 0$.
Therefore Property (iv) follows.
It thus suffices to prove Claim~\ref{clm:claim0}.

\begin{proof}[Proof of Claim~\ref{clm:claim0}]
Recall that $\mu$ consists of three terms.
For the second term, by linearity of expectation
\begin{align} \label{eq:bip_bound_1}
 	\BBE\left[|\hat{B}_2 \cap A_2'|\right]
 	= \sum_{x \in A_2'} \BFP\left(x \in N(\hat{T}_1)\right)
 	= \sum_{x \in A_2'} \left(\frac{\deg(x; V_1')}{|V_1'|} \right)^{s}.
\end{align}
Therefore by convexity,
\begin{align} \label{eq:bip_bound_2}
 	\BBE\left[|\hat{B}_2 \cap A_2'|\right]
 	\ge |A_2'| \left(\frac{1}{|A_2'|} \sum_{x \in A_2'} \frac{\deg(x; V_1')}{|V_1'|} \right)^{s}
 	\ge \left(\frac{\delta}{2} - \varepsilon\right)^{s} |A_2'|,
\end{align}
where the second inequality follows since all vertices in $A_2'$ have degree at least
$\frac{\delta}{2}|V_1|$ and $|V_1'| \ge (1-\varepsilon)|V_1|$.

For the first term of $\mu$, by linearity of expectation,
\begin{align} \label{eq:intersection_size}
	\BBE\left[ \big| \hat{B}_2 \cap A_2' \big| \cdot \big| \hat{B}_2' \big|  \right]
	= &\, \sum_{x \in A_2'}\sum_{x' \in V_2'} \BFP\left(x, x' \in N(\hat{T}_1)\right) \nonumber \\
	=&\, \sum_{x \in A_2'}\sum_{x' \in V_2'} \left(\frac{\codeg(x,x'; V_1')}{|V_1'|}\right)^{s} \nonumber \\
	\ge&\, \sum_{x \in A_2'}|V_2'| \left(\frac{1}{|V_2'|}\sum_{x' \in V_2'} \frac{\codeg(x,x'; V_1')}{|V_1'|}\right)^{s},
\end{align}
where the final inequality follows from convexity.
For fixed $x \in A_2'$, the sum $\sum_{x' \in V_2'} \codeg(x,x'; V_1')$ counts the number of paths of length 2 starting at $x$ whose second vertex is in $V_1'$ and third in $V_2'$. Therefore by the minimum degree condition we have
\[
	\sum_{x' \in V_2'} \codeg(x,x'; V_1')
	\ge \deg(x; V_1') \cdot (\delta|V_2| - \varepsilon|V_2|).
\]
Hence from \eqref{eq:intersection_size} and \eqref{eq:bip_bound_1},
\begin{align} \label{eq:mu_1}
	\BBE\left[ \big| \hat{B}_2 \cap A_2 \big| \cdot \big| \hat{B}_2' \big|  \right]
	\ge \sum_{x \in A_2'}|V_2'| (\delta - \varepsilon)^{s} \left( \frac{\deg(x;V_1')}{|V_1'|}\right)^{s}
	= (\delta - \varepsilon)^s |V_2'| \cdot \BBE\left[ |\hat{B}_2 \cap A_2'|\right].
\end{align}
By Proposition \ref{prop:negligible_potential} (i), 
\[
	\BBE\left[\frac{n^{2} }{\lambda^{2s-1}}\eta_{2s}(\hat{B}_2, V_1)\right]
	\le \frac{n^{2} }{\lambda^{2s-1}} \eta_{2s}(V_2, V_1) \cdot \left(\frac{\beta}{|V_1'|}\right)^{s}.
\]
Observe that $\eta_{2s}(V_2, V_1) < n^{2s+d}$ trivially holds since $\eta_{2s}(V_2, V_1)$ 
counts the number of $(2s+d)$-tuples in $V_2$ with less than $2\beta$ common neighbors in $V_1$.
Recall that $s=\sqrt{\frac{d \log n}{\log(2/\delta)}}$, $\lambda \le \left(\frac{\delta}{2}\right)^{5s}n$, 
$\beta \le \left(\frac{\lambda}{n}\right)^{3}\gamma n$, and $|V_1'| \ge (1-\varepsilon)\gamma n$. Hence
\begin{align*}
	\BBE\left[\frac{n^{2} }{\lambda^{2s-1}}\eta_{2s}(\hat{B}_2, V_1)\right]
	&\le \frac{n^{2s+d+2} }{\lambda^{2s-1}} \left(\frac{\beta}{|V_1'|}\right)^{s}
	\le \frac{\lambda n^{d+2}}{(1-\varepsilon)^s \gamma^s} \cdot \left(\frac{n^2}{\lambda^2} \cdot \frac{\beta}{n}\right)^{s}
\\	&\le \frac{\lambda n^{d+2}}{(1-\varepsilon)^s} \cdot \left( \frac{\lambda}{n} \right)^{s}
	= \frac{\lambda n^{d+2}}{(1-\varepsilon)^s} \cdot \left( \frac{\delta}{2} \right)^{5s^2}
	= \frac{\lambda n^{-4d+2}}{(1-\varepsilon)^s}
	< 1.
\end{align*}
Therefore from \eqref{eq:mu_1}, \eqref{eq:bip_bound_2} and $\varepsilon < \frac{\delta}{8}$, we obtain
\begin{align*}
	\mu 
	=&\,
	\BBE\left[ \big| \hat{B}_2 \cap A_2 \big| \cdot \big| \hat{B}_2' \big|
	- \left(\frac{\delta}{2}\right)^{s} |V_2'| \cdot |\hat{B}_2 \cap A_2'|
	- \frac{n^{2} }{\lambda^{2s-1}}\eta_{2s}(\hat{B}_2, V_1)\right] \\
	\ge&\, \frac{1}{2}(\delta-\varepsilon)^{s} |V_2'| \cdot \BBE\left[|\hat{B}_2 \cap A_2'|\right] \\
	\ge&\, \left( \frac{\delta}{2}\right)^{2s} |V_2'||A_2'|
	\ge \frac{1}{8} \left( \frac{\delta}{2}\right)^{4s} |V_2|^2.
\end{align*}

We now compute the expected value of $\nu$. The choice of $B_2$ in particular implies that
\[
	\big| B_2 \cap A_2' \big| \cdot \big| B_2' \big|
	\ge \frac{1}{8} \left( \frac{\delta}{2}\right)^{4s} |V_2|^2,
\]
and since $|B_2'| \le |V_2|$, it follows that
$|B_2 \cap A_2'| \ge \frac{1}{8} \left( \frac{\delta}{2}\right)^{4s} |V_2|$.
Recall that
\[
	\nu = \BBE\left[ \eta_{0}(\hat{B}_1, A_2 \cap B_2)  + \frac{\lambda^{s}\eta_{s}(B_2, \hat{B}_1)}{\eta_{2s}(B_2, V_1)} + \frac{\lambda^{s}\xi_{0}(A_2, \hat{B}_1 \cap A_1)}{\xi_{s}(A_2, A_1)} \right].
\]
By Proposition \ref{prop:negligible_potential} (i),
since $\beta \le \left(\frac{\lambda}{n}\right)^{3}\gamma n \le \left(\frac{\delta}{2}\right)^{15s}\gamma n$,
\begin{align*}
	\BBE\left[\eta_{0}(\hat{B}_1, A_2 \cap B_2)\right]
	\le&\, \eta_{0}(V_1, A_2 \cap B_2) \cdot \left( \frac{2\beta}{|B_2 \cap A_2'|} \right)^{s} \\
	<&\, n^d \left( \frac{2\beta}{|B_2 \cap A_2'|} \right)^{s} 
	\le n^d \left( \frac{\delta}{2} \right)^{10s^2} = n^{-9d}
	< \frac{1}{2}.
\end{align*}
By Proposition \ref{prop:negligible_potential} (ii),
\[
	\lambda^{s} \cdot \BBE\left[\frac{\xi_{0}(A_2, \hat{B}_1 \cap A_1)}{\xi_{s}(A_2, A_1)} + \frac{\eta_{s}(B_2, \hat{B}_1)}{\eta_{2s}(B_2, V_1)} \right]
	\le 2 \left( \frac{\lambda}{|B_2 \cap A_2'|}\right)^{s} < \frac{1}{2},
\]
and thus $\nu < 1$.
\end{proof}

The notations $V_1', V_2', A_2', B_2'$ used in this subsection are somewhat cumbersome but plays
an important role in the proof. Recall that
\[
	\mu = \BBE\left[ |\hat{B}_2 \cap A_2'| \cdot |\hat{B}_2'| - \left(\frac{\delta}{2}\right)^{s} |V_2'| \cdot |\hat{B}_2 \cap A_2'| - \frac{n^{2} }{\lambda^{2s}}\eta_{2s}(\hat{B}_2, V_1) \right].
\]
The second term $\left(\frac{\delta}{2}\right)^{s} |V_2'| \cdot |\hat{B}_2 \cap A_2'|$
is extremely important since it allows us to give a lower bound on the size of $B_2'$ 
that is independent of the size of $A_2'$. 
The reason we were able to conclude that $\mu$ is positive even with this term is 
\eqref{eq:mu_1} which were deduced from the fact that we are working with vertices
in $V_2'$, not $V_2$.
The same phenomenon shows up in greater depth for non-bipartite graphs.
Thus it is useful to identify these sets of vertices having large degree.

It may seem like one can avoid using these notations.
One way is by
defining $B_2$ as some subset of $N(T_1)$ (instead of $B_2 = N(T_1)$). However such approach
fails since if $B_2 \subsetneq N(T_1)$ then $T_1 \cup Q$ having many common neighbors 
no longer implies that $Q$
has many common neighbors in $B_2$ (we crucially relied on such property throughout the proof).
Another way is by imposing the condition that
all vertices in $A_2$ have relative degree at least $\delta$ in $A$, and prove
that all vertices in $B_2$ have relative degree at least $\delta$ in $A$ as well.
The problem with this approach is that we embed $\beta$ vertices of $H$ 
after constructing the sets $A_1$ and $A_2$. Once we embed these vertices, the set $A_2$
has relative degree at least $\delta'$ in the remaining set for some $\delta' < \delta$.
Since $B_2$ necessarily intersects $A_2$, this implies that $B_2$ 
can only be guaranteed to have relative degree at least $\delta'$ (instead of $\delta$).
Moreover, after embedding another $\beta$ vertices, it will drop to relative
degree at least $\delta''$ for some $\delta'' < \delta$. 
Therefore the embedding algorithm becomes unsustainable under this approach.

\section{General graphs} 
\label{sec:general}

Throughout this section, fix an $r$-partite graph $H$.
We will consider the problem of embedding $H$ into another $r$-partite graph $G$.
We assume that $r$-partitions $(W_i)_{i \in [r]}$ of $H$ 
and $(V_i)_{i \in [r]}$ of $G$ 
are given. We restrict our attention
to finding an embedding $f : V(H) \rightarrow V(G)$ for which $f(W_i) \subseteq V_i$ for all $i \in [r]$
even when we do not explicitly refer to the $r$-partition of $H$.

The embedding lemma that we use for $r$-partite graphs
is in fact simpler than the one used in the previous section (Lemma \ref{lem:extend_partial}). Such simplicity is achieved at the cost of having further restriction on the given $r$-tuple
$(A_i)_{i \in [r]}$.
Throughout this section, for an $r$-tuple of sets $\mathcal{A} = (A_i)_{i \in [r]}$, 
we use the notation $A_{-i}$ to denote the set $A_{-i} := \bigcup_{j \in [r], j \neq i} A_j$.
We say that $\mathcal{A}$ is {\it $(d,\beta)$-common}
if $A_{-i}$ is $(d,\beta)$-common into $A_i$ for every $i \in [r]$.
We say that an $r$-tuple $\mathcal{T} = (T_i)_{i \in [r]}$ is
{\it $(d,\beta)$-typical for $\mathcal{A}$} if for every $i \in [r]$ and
$d$-tuple of vertices $Q$ in $A_{-i}$, the set $Q \cup T_{-i}$
has at least $\beta$ common neighbors in $A_{i}$, 
or equivalently, $Q$ has at least $\beta$ common neighbors in $A_{i} \cap N(T_{-i})$.
Note that if $(T_i)_{i \in [r]}$ is $(d,\beta)$-typical for $(A_i)_{i \in [r]}$,
then $A_{-i}$ is $(d,\beta)$-common into $A_i \cap N(T_{-i})$ for all $i \in [r]$.

\begin{lem} \label{lem:rolling_embedding_g}
Let $G$ be an $r$-partite graph with vertex partition $(A_i)_{i \in [r]}$.
Let $H$ be an $r$-partite graph with vertex partition $(W_i)_{i \in [r]}$ and a 
$d$-degenerate labelling by $[\beta'+\beta]$.
Suppose that an $r$-tuple of sets $(T_i)_{i \in [r]}$ is $(d,\beta'+\beta)$-typical for $(A_i)_{i \in [r]}$.
Then every partial embedding $f : [\beta'] \rightarrow \bigcup_{i \in [r]} A_i$ 
can be extended to an embedding $f : [\beta'+\beta] \rightarrow \bigcup_{i \in [r]} A_i$ where all vertices with label greater than $\beta'$ are mapped into $\bigcup_{i \in [r]} (A_i \cap N(T_{-i}))$.
\end{lem}
\begin{proof}
As mentioned above, the given condition implies that $A_{-i}$ is $(d,\beta'+\beta)$-common into 
$A_i' := A_i \cap N(T_{-i})$.
We will extend the partial embedding $f$ one vertex at a time according to the 
order given by the vertex labelling.
Suppose that for some $t > \beta'$, 
the partial embedding has been defined over $[t-1]$, and we are about to embed vertex $t \in W_j$
for some index $j \in [r]$. 
There are at most $d$ neighbors of $t$ that precede $t$. 
Call these vertices $v_1, v_2, \ldots, v_{d'}$ (for $d' \le d$), and let $w_i = f(v_i)$ be their images.
Note that $\{w_i\}_{i \in [d']} \subseteq A_{-j}$. 
Since $A_{-j}$ is $(d,\beta'+\beta)$-common into $A_{j}'$, the vertices $\{w_i\}_{i \in [d']}$
have at least $\beta'+\beta$ common neighbors in $A_{j}'$. 
Since $H$ consists of $\beta'+\beta$ vertices, 
there exists at least one vertex $x \in A_{j}'$ sadjacent to all vertices $w_1, w_2, \ldots, w_{d'}$
and is not an image of $f$. By defining $f(t) = x$, we can extend $f$ to $[t]$.
We can find the desired extension $f$ by repeating this process.
\end{proof}

In order to find sets that satisfy the conditions of Lemma \ref{lem:rolling_embedding_g},
we need to maintain control on potentials of sets and the number of copies of $K_r$
across given family of sets.
For an $r$-tuple of sets $\mathcal{A} = (A_i)_{i \in [r]}$, we say that 
a copy of $K_r$ is {\it across $\mathcal{A}$} if it has one vertex in each set $A_i$ for $i \in [r]$. 
Let $G$ be an $r$-partite graph with vertex partition $(V_i)_{i \in [r]}$.
A copy of $K_r$ across $(V_i)_{i \in [r]}$ with vertex set $(v_i)_{i \in [r]}$ is {\it $\delta$-heavy} 
if for each $j \in [r]$, the $(r-1)$-tuple of vertices $(v_i)_{i \in [r]\setminus \{j\}}$
has at least $\delta |V_j|$ common neighbors in $V_j$.
Define $\kappa_\delta(\mathcal{A})$ as
the number of $\delta$-heavy copies of $K_r$ across $\mathcal{A}$.
Sometimes we additionally denote the graph $G$ such as in $\kappa_{\delta}(\mathcal{A}; G)$ 
to specify that we are counting copies of $K_r$ in $G$.
Throughout the proof, we carefully keep track of $\delta$-heavy copies of $K_r$. 
This corresponds to the fact that
we were keeping track of the vertices of large degree in the previous section.

Recall that for a pair of sets $X$ and $Y$,
the $(p,d,\beta)$-potential of $X$ in $Y$ is the number of 
$(p+d)$-tuples in $X$ that have fewer than $\beta$ common neighbors in $Y$.
Further recall that $X$ has $\lambda$-negligible $(p,d,\beta)$-potential in $Y$
if its $(p,d,\beta)$-potential in $Y$ is less than $\lambda^{p-1}$.
For an $r$-tuple of sets $\mathcal{A} = (A_i)_{i \in [r]}$, 
we say that {\em $\mathcal{A}$ has $\lambda$-negligible $(p,d,\beta)$-potential} 
if $A_{-i}$ has $\lambda$-negligible $(p,d,\beta)$-potential in $A_i$ for all $i \in [r]$.
The following simple fact about potentials will be useful.

\begin{prop} \label{prop:potential_monotone}
If $X$ has $\lambda$-negligible $(p,d,\beta)$-potential in $Y$ for some $\lambda \le |X|$,
then it has $\lambda$-negligible $(p',d,\beta)$-potential in $Y$ for all $p' \le p$.
\end{prop}
\begin{proof}
Let $T$ be a $p'$-tuple of vertices in $X$ with less than $\beta$ common neighbors in $Y$.
Then for every $(p-p')$-tuple $T'$, the $p$-tuple $T \cup T'$ has less than $\beta$ common neighbors in $Y$.
Therefore
\[
	\xi_{p',d,\beta}(X,Y) |X|^{p - p'}
	\le \xi_{p,d,\beta}(X,Y) 
	< \lambda^{p-1}.
\]
Since $|X| \ge \lambda$, it follows that $\xi_{p',d,\beta}(X,Y) < \lambda^{p'-1}$.
\end{proof}

The next lemma, corresponding to Lemma~\ref{lem:rolling_drc_degenerate} in the
bipartite case, produces sets satisfying the condition of Lemma \ref{lem:rolling_embedding_g}.
As in the previous section, we defer its proof into another subsection.

\begin{lem} \label{lem:rolling_drc_degenerate_g}
Let $\varepsilon, \delta, r$ be fixed parameters and $n$ be a sufficiently large integer.
Let $G$ be an $r$-partite $(\varepsilon,\delta)$-dense graph with vertex partition $(V_i)_{i \in [r]}$,
and $\mathcal{A} = (A_i)_{i \in [r]}$ be an $r$-tuples of sets with $A_i \subseteq V_i$ for all $i \in [r]$.
For $s = \left( \frac{d \log n}{\log \log n}  \right)^{1/(2r)}$, suppose that the following conditions hold,
\begin{itemize}
  \setlength{\itemsep}{1pt} \setlength{\parskip}{0pt}
  \setlength{\parsep}{0pt}
  \item[(a)] $\varepsilon \le \left(\frac{\delta}{2}\right)^{2r}$,
	$\frac{\lambda}{n} \le \left(\frac{n_0}{n}\right)^{2} \left(\frac{\delta}{\log n}\right)^{5r^2 (10s)^{2r-1}}$, and
	$\frac{\beta}{n} \le \left(\frac{\lambda}{n}\right)^{2r}$,
  \item[(b)] $n_0 \le |V_i| \le n$ for all $i \in [r]$,
  \item[(c)] $\kappa_{(\delta/2)^r}(\mathcal{A}) \ge \frac{1}{2}\left(\frac{\delta}{\log n}\right)^{2r^2 (10s)^r} \prod_{i=1}^{r} |V_i|$.
\end{itemize}
Then there exists an $r$-tuple $\mathcal{T} = (T_i)_{i \in [r]}$ satisfying the following properties:
\begin{itemize}
  \setlength{\itemsep}{1pt} \setlength{\parskip}{0pt}
  \setlength{\parsep}{0pt}
  \item[(i)] $\mathcal{B} := (V_i \cap N(T_{-i}))_{i \in [r]}$ is $\big(d,2\beta)$-common,
  \item[(ii)] $\mathcal{B}$ has $\lambda$-negligible $\left((r-1)s, d, 2\beta\right)$-potential, and
  \item[(iii)] $\kappa_{\delta^r}(\mathcal{B}) \ge \left(\frac{\delta}{\log n}\right)^{2r^2(10s)^r} \prod_{i=1}^{r} |V_i|$.
\end{itemize}
Furthermore , if $\mathcal{A}$ has $\lambda$-negligible $\left((r-1)s, d, \beta\right)$-potential, then
\begin{itemize}
  \item[(iv)] $\mathcal{T}$ is $(d,\beta)$-typical for $\mathcal{A}$.
\end{itemize}
\end{lem}

We now prove Theorem \ref {thm:main} in a slightly more general form as stated below,
using the tools developed in this subsection.

\begin{thm} \label{thm:main_extend}
For fixed $\varepsilon, \delta, r$  satisfying $\varepsilon \le (\frac{\delta}{2})^{2r}$, there exists $c$ such that the following holds.
Suppose that $\beta_0 \le e^{-c(d_0 \log n)^{(2r-1)/2r} (\log \log n)^{1/2r}} \gamma^{4r} n$
and $n$ is sufficiently large.
Let $H$ be a $d_0$-degenerate graph with bandwidth $\beta_0$
and an $r$-partition $(W_i)_{i \in [r]}$.
If $G$ is an $n$-vertex $(\varepsilon^2,\delta)$-dense
$r$-partite graph with vertex partition $(V_i)_{i \in [r]}$ satisfying 
$|V_i| \ge \max\{\frac{1}{1-\varepsilon}|W_i|, \gamma n\}$ for all $i \in [r]$, 
then $G$ contains a copy of $H$.
\end{thm}
\begin{proof}
Define $s = \left( \frac{d\log n}{\log \log n}  \right)^{1/2r}$,
$\lambda = \left(\frac{\delta}{\log n}\right)^{5r^2(10s)^{2r-1}}\varepsilon^2 \gamma^2 n$,
and $n_0 = \varepsilon \gamma n$.
Let $H$ be a given $m$-vertex $d_0$-degenerate $r$-partite graph of bandwidth $\beta_0$
with vertex partition $(W_i)_{i \in [r]}$.
By Lemma~\ref{lem:bandwidth_degenerate}, there exists a labelling of the vertex
set of $H$ by $[m]$ in which all adjacent pairs are at most $\beta := \beta_0 \log_2 (4\beta_0)$
apart and all vertices have at most $d := 5d_0$ neighbors preceding itself.
For sufficiently large $c$ depending on $\varepsilon, \delta, r$, the given bound on 
$\beta_0$ implies
$\frac{2\beta}{n} \le \left(\frac{\lambda}{n}\right)^{2r}$
for large enough $n$.

Let $G$ be a given $n$-vertex $r$-partite graph on $V := V(G)$ 
with vertex partition $(V_i)_{i \in [r]}$, where
$|W_i| \le (1-\varepsilon)|V_i|$ for all $i \in [r]$.
Our goal is to find an embedding $f$ of $H$
into $G$ for which $f(W_i) \subseteq V_i$ for all $i \in [r]$.
Define $I_t = ((t-1)\beta, t\beta] \cap [m]$ for $t \ge 1$.
We will embed $H$ using an iterative algorithm. 
For $t \ge 1$, at the beginning of the $t$-th step, we are given an 
$r$-tuple of sets $\mathcal{A} = (A_i)_{i \in [r]}$ 
and a partial embedding $f : [(t-1)\beta] \rightarrow V(H)$ satisfying 
the following properties. Define $V_i^{(t)} := V_i \setminus f([(t-2)\beta])$ for all $i \in [r]$.
\begin{itemize}
  \setlength{\itemsep}{1pt} \setlength{\parskip}{0pt}
  \setlength{\parsep}{0pt}
  \item[(i)] $A_i \subseteq V_i^{(t)}$ for all $i \in [r]$,
  \item[(ii)] $f(I_{t-1}) \subseteq \bigcup_{i \in [r]} A_i$.
  \item[(iii)] $\mathcal{A}$ has $\lambda$-negligible $((r-1)s, d, 2\beta)$-potential, and
  \item[(iv)] $\kappa_{(\delta/2)^{r}}(\mathcal{A}; G^{(t)}) \ge \frac{1}{2}\left(\frac{\delta}{\log n}\right)^{2r^2(10s)^{r}} \prod_{i=1}^{r} |V_i^{(t)}|$ where $G^{(t)}$ is the subgraph of $G$ induced on $\bigcup_{i \in [r]} V_i^{(t)}$.
\end{itemize}
Given $\mathcal{A}$ and $f$ as above, the $t$-th step of the iteration finds an $r$-tuple of sets 
$\mathcal{B} = (B_i)_{i \in [r]}$ and extends $f$ so that 
the conditions for the $(t+1)$-th step is satisfied (where $\mathcal{A}$ is replaced with $\mathcal{B}$).
The process can then be continued until we find an embedding of $H$ into $G$.

As an intial step, apply Lemma \ref{lem:rolling_drc_degenerate_g}
to $\mathcal{A}_{\ref{lem:rolling_drc_degenerate_g}} = (V_i)_{i \in [r]}$ 
to find an $r$-tuple of sets $\mathcal{A}^{(0)} = (A_i^{(0)})_{i \in [r]}$ 
that is $(d, 2\beta)$-common, has $\lambda$-negligible $((r-1)s, d, 4\beta)$-potential, and satisfies
$\kappa_{\delta^r}(\mathcal{A}^{(0)}) \ge \left(\frac{\delta}{\log n}\right)^{2r^2(10s)^r} \prod_{i=1}^{r} |V_i|$.
This provides the $r$-tuple needed to begin the step $t=1$ (where $f$ is the
partial embedding defined on the empty-set).

Now suppose that for some $t \ge 1$, we completed the $(t-1)$-th step of the algorithm
and are given $\mathcal{A}$ and $f$ satisfying the properties listed above.
Since $G$ is $(\varepsilon^2, \delta)$-dense and $|W_i| \le (1-\varepsilon)|V_i|$
for all $i \in [r]$, the subgraph $G^{(t)}$ is $(\varepsilon, \delta)$-dense. Thus we can apply Lemma \ref{lem:rolling_drc_degenerate_g}
to the graph $G^{(t)}$ 
with $(n_0)_{\ref{lem:rolling_drc_degenerate_g}} = \varepsilon n_0$,
$\beta_{\ref{lem:rolling_drc_degenerate_g}} = 2\beta$, 
and $\mathcal{A}_{\ref{lem:rolling_drc_degenerate_g}} = \mathcal{A}$, and obtain
an $r$-tuple $\mathcal{T} = (T_i)_{i \in [r]}$ for which $C_i := N(T_{-i}) \cap V_i$
and $\mathcal{C} := (C_i)_{i \in [r]}$ satisfies the following properties:
\begin{itemize}
  \setlength{\itemsep}{1pt} \setlength{\parskip}{0pt}
  \setlength{\parsep}{0pt}
  \item[(i')] $\mathcal{C}$ is $\big(d,4\beta)$-common,
  \item[(ii')] $\mathcal{C}$ has $\lambda$-negligible $\left((r-1)s, d, 4\beta\right)$-potential, 
  \item[(iii')] $\kappa_{\delta^{r}}(\mathcal{C}; G^{(t)}) \ge \left(\frac{\delta}{\log n}\right)^{2r^2(10s)^r} \prod_{i=1}^{r} |V_i^{(t)}|$, and
  \item[(iv')] $\mathcal{T}$ is $(d,2\beta)$-typical for $\mathcal{A}$.
\end{itemize}

Consider $f|_{I_{t-1}}$ as a partial embedding of $H[I_{t-1} \cup I_{t}]$ defined
on $I_{t-1}$ and apply Lemma~\ref{lem:rolling_embedding_g} to extend it to $I_{t}$. 
Since $H$ has no edge between $I_{t}$ and $[(t-2)\beta]$, it gives an extension of $f$
to $I_t$ so that $f(I_t) \subseteq \bigcup_{i \in [r]} (A_i \cap C_i) \subseteq \bigcup_{i \in [r]} C_i$.
Define $B_i := C_i \setminus f(I_{t-1})$ for each $i \in [r]$ and denote $\mathcal{B} = (B_i)_{i \in [r]}$.
Since $C_i \subseteq V_i^{(t)}$, it follows that $B_i \subseteq V_i^{(t+1)}$ for all $i \in [r]$.
Thus $\mathcal{B}$ satisfies Property~(i) for the $t$-th step.
It also satisfies Property~(ii) since $f(I_t)$ and $f(I_{t-1})$ are disjoint by 
the definition of embedding and $f(I_t) \subseteq \bigcup_{i \in [r]} C_i$.

Note that $G^{(t+1)}$ is obtained from $G^{(t)}$ by removing $\beta$ vertices.
Hence for a set of vertices $X$, its $(p,d,2\beta)$-potential in $B_i$ is at most
its $(p,d,3\beta)$-potential in $C_i$, which is at most its $(p,d,4\beta)$-potential in $C_i$, for all $i \in [r]$.
Thus Property (iii) follows from Property (ii').
Since $\delta^{r} \varepsilon n_0 - \beta \ge (\frac{\delta}{2})^{r} \varepsilon n_0$,  
every $\delta^{r}$-heavy 
copy of $K_r$ in $G^{(t)}$ is $(\frac{\delta}{2})^{r}$-heavy in $G^{(t+1)}$. 
Therefore
\[
	\kappa_{(\delta/2)^{r}}(\mathcal{B}; G^{(t+1)})
	\ge \kappa_{\delta^{r}}(\mathcal{B}; G^{(t)}) - \frac{\beta}{n} \prod_{i=1}^{r} |V_i^{(t)}|
	\ge \frac{1}{2}\left(\frac{\delta}{\log n}\right)^{2r^2(10s)^r} \prod_{i=1}^{r} |V_i^{(t+1)}|,
\]
which verifies Property (iv) and concludes the proof.
\end{proof}

\subsection{Proof of Lemma \ref{lem:rolling_drc_degenerate_g}}

In this subsection, we prove Lemma~\ref{lem:rolling_drc_degenerate_g} through a 
slightly more general lemma applicable to later applications as well.

\subsubsection{Generalized version}

For an $r$-tuple of sets $\mathcal{A} = (A_i)_{i \in [r]}$, we say that 
an $r$-partite graph $F$ is {\em across $\mathcal{A}$}, or {\em crosses $\mathcal{A}$},
if the $i$-th part of $F$ is in $A_i$ for all $i \in [r]$.
For a family $\mathcal{F}$ of $r$-partite subgraphs of $G$, we use the notation
$\mathcal{F}(\mathcal{A})$ to denote the subfamily of graphs in $\mathcal{F}$ that
cross $\mathcal{A}$.
For two graphs $F_1$ and $F_2$ that cross $\mathcal{A}$, 
we say that $F_1$ and $F_2$ are {\it adjacent} if for every distinct
$j, j' \in [r]$, all vertices in the $j$-th part of $F_1$ is adjacent 
to all the vertices of the $j'$-th part of $F_2$.
For a family $\mathcal{F}$ of graphs that cross $\mathcal{A}$,
we say that a graph $F$ is {\em $\delta$-heavy with respect to $\mathcal{F}$}
if there are at least $\delta|\mathcal{F}|$ graphs in $\mathcal{F}$
adjacent to $F$. For another family of graphs $\mathcal{F}'$, we say 
that $\mathcal{F}'$ is {\em $\delta$-heavy with respect to $\mathcal{F}$} if 
each graph in $\mathcal{F}'$ is $\delta$-heavy with respect to $\mathcal{F}$.
Recall that a copy $K$ of $K_r$ across an $r$-partition $(V_i)_{i \in [r]}$ 
whose vertices are given by $(v_i)_{i \in [r]}$
is $\delta$-heavy if for each $j \in [r]$, 
the vertices $(v_i)_{i \in [r] \setminus \{j\}}$ have at least $\delta|V_i|$
common neighbors in $V_i$. Thus if $K$ is $\delta$-heavy, then it is
$\delta^r$-heavy with respect to the family of all copies of the empty $r$-vertex
graph across the partition.

The following lemma implies Lemma~\ref{lem:rolling_drc_degenerate_g}.

\begin{lem} \label{lem:rolling_drc_general}
Suppose that parameters satisfying the following conditions are given:
\[
	\left(\frac{\beta}{n}\right)^{1/(c+r+1)}
	\le
	\frac{\lambda}{n} 
	\le 
	\min\left\{\left(\frac{n_0}{n}\right)^{2}\left(\frac{\delta_1 \delta_2}{\log^2 n}\right)^{2(10s)^{r-1}},
	\left(\frac{1}{2rn^{d+r+f+1}}\right)^{1/s}
	 \right\},
\]
where $n$ is sufficiently large.
Let $G$ be an $r$-partite graph with vertex partition $(V_i)_{i \in [r]}$ satisfying
$n_0 \le |V_i| \le n$ for all $i \in [r]$.
Let $F_0$ be a $f$-vertex $r$-partite graph and $\mathcal{F}$ be a family of copies 
of $F_0$ across $(V_i)_{i \in [r]}$. 
Let $\mathcal{K}$ be a family of copies of $K_r$ that is $\delta_1$-heavy with respect to $\mathcal{F}$.
Suppose that an $r$-tuple of sets $\mathcal{A} = (A_i)_{i \in [r]}$ with $A_i \subseteq V_i$ satisfy
$|\mathcal{K}(\mathcal{A})| \ge \delta_2 \prod_{i=1}^{r} |V_i|$.
Then there exists an $r$-tuple $\mathcal{T} = (T_i)_{i \in [r]}$ satisfying the following properties:
\begin{itemize}
  \setlength{\itemsep}{1pt} \setlength{\parskip}{0pt}
  \setlength{\parsep}{0pt}
  \item[(i)] $\mathcal{B} := (V_i \cap N(T_{-i}))_{i \in [r]}$ has $\lambda$-negligible $\left(s', d, 2\beta\right)$-potential for all $s' \le c s$, and
  \item[(ii)] $|\mathcal{F}(\mathcal{B})| \ge \left(\frac{\delta_1}{\log^2 n}\right)^{(10s)^r} |\mathcal{F}|$.
\end{itemize}
Furthermore if $\mathcal{A}$ has $\lambda$-negligible $\left((r-1)s, d, \beta\right)$-potential, then
\begin{itemize}
  \setlength{\itemsep}{1pt} \setlength{\parskip}{0pt}
  \setlength{\parsep}{0pt}
  \item[(iii)] $\mathcal{T}$ is $(d,\beta)$-typical for $\mathcal{A}$.
\end{itemize}
\end{lem}

The reason we develop this more general version, even though 
we use Lemma \ref{lem:rolling_drc_general} only with $F = K_r$ and
$\mathcal{F}$ being the family of $\delta$-heavy copies of $K_r$ in this section,
is for the applications that will be given in the next section in which
we use different graphs $F$.
Note that Property (ii) gives a bound not depending on $\delta_2$. This is a crucial feature
since the lemma will be repeatedly used where the bound obtained in Property (ii)
will decide the value of $\delta_2$ in the next iteration. Thus if the bound in Property (ii)
depended on $\delta_2$, then it will rapidly shrink over iterations, and thus 
we will only be able to repeat this lemma for a few times.

In order to deduce Lemma~\ref{lem:rolling_drc_degenerate_g} from Lemma~\ref{lem:rolling_drc_general},
we will use the following variant of the well-known counting lemma.

\begin{lem} \label{lem:counting}
Let $r \ge 2$ be an integer and 
$\varepsilon, \delta, \delta'$ be positive real numbers satisfying 
$0< \varepsilon < \delta' \delta^{r}$ and $\delta \le \frac{1}{2}$.
Let $G$ be an $(\varepsilon, \delta)$-dense $r$-partite graph with vertex partition
$(V_i)_{i \in [r]}$ and let $K$ be a $\delta'$-heavy copy of $K_r$ across the partition.
Then there exists at least $(\delta')^{r} \delta^{r^2} \prod_{i=1}^{r} |V_i|$ copies of $K_r$ across the partition
that are $\delta^{r}$-heavy and adjacent to $K$.
\end{lem}
\begin{proof}
Let $(v_i)_{i \in [r]}$ be the vertices of $K$ where $v_i \in V_i$ for all $i \in [r]$.
Define $W_i = V_i \cap \bigcap_{j \in [r] \setminus \{i\}} N(v_j)$. Since $K$ is $\delta'$-heavy,
we see that $|W_i| \ge \delta' |V_i|$ for all $i \in [r]$.

We will find copies of $K_r$ across the partition by choosing one vertex at a time.
More precisely for $t \in [r]$, at the $t$-th step we will choose a vertex $w_t \in W_t$ satisfying
the following properties:
\begin{itemize}
  \setlength{\itemsep}{1pt} \setlength{\parskip}{0pt}
  \setlength{\parsep}{0pt}
  \item[(i)] $w_t$ is adjacent to $w_1, \cdots, w_{t-1}$, 
  \item[(ii)] For $i \in [r]$, the set $W_i^{(t)} := W_i \cap \bigcap_{j \le t, j \neq i} N(w_j)$ has size at least  $\delta^{t}|W_i|$.
\end{itemize}
We will prove that there are at least $\delta' \delta^r|V_t|$ choices 
for the vertex $w_t$ satisfying the above at each step. Note that this implies the lemma.

Suppose that for some $t \ge 1$ we have already found $t-1$ vertices $w_1, \cdots, w_{t-1}$
satisfying the above. Consider the sets $W_i^{(t-1)} \subseteq W_i$ for $i \in [r]$ defined as above. 
Note that all vertices in $W_t^{(t-1)}$ are adjacent to the vertices $w_1, \cdots, w_{t-1}$.
Further note that $|W_i^{(t-1)}| \ge \delta^{t-1}|W_i| \ge \varepsilon|V_i|$ 
for all $i \in [r]$.
Since $G$ is $(\varepsilon,\delta)$-dense, there are at most
$\varepsilon|W_t^{(t-1)}|$ vertices in $W_t^{(t-1)}$ that have less than 
$\delta |W_i^{(t-1)}|$ neighbors in $W_i^{(t-1)}$.
Since $\varepsilon r < \frac{1}{2}$, at least $\frac{1}{2}|W_t^{(t-1)}|$ vertices in $W_t^{(t-1)}$ 
 have at least $\delta |W_i^{(t-1)}|$ neighbors in $W_i^{(t-1)}$ 
for all $i \in [r] \setminus \{t\}$.
All these vertices can be used
as the vertex $w_t$ and would satisfy the two properties listed above.
Since 
\[
	\frac{1}{2}|W_t^{(t-1)}|
	\ge \frac{1}{2} \delta^{t-1} |W_t|
	\ge \delta' \delta^{r} |V_t|,
\]
this concludes the proof.
\end{proof}

Lemma \ref{lem:rolling_drc_degenerate_g} straightforwardly follows from
Lemmas \ref{lem:rolling_drc_general} and \ref{lem:counting}.

\begin{proof}[Proof of Lemma \ref{lem:rolling_drc_degenerate_g}]
Define $\delta_1 = (\frac{\delta^2}{2})^{r^2}$ and $\delta_2 = \frac{1}{2} \left(\frac{\delta}{\log n}\right)^{2r^2(10s)^r}$.
Suppose that we are given a partition $(V_i)_{i \in [r]}$ and an $r$-tuple
of sets $(A_i)_{i \in [r]}$ satisfying $A_i \subseteq V_i$ for all $i \in [r]$.
Let $\mathcal{K}$ be the family of $\left(\frac{\delta}{2}\right)^{r}$-heavy
copies of $K_r$ across the partition $(V_i)_{i \in [r]}$.
The given condition implies $|\mathcal{K}(\mathcal{A})| \ge \delta_2 \prod_{i=1}^{r} |V_i|$.
Define $\mathcal{F}$ as the family of $\delta^{r}$-heavy 
copies of $K_r$ across the partition $(V_i)_{i \in [r]}$.
Since $\varepsilon \le \left(\frac{\delta}{2}\right)^{2r}$,
Lemma \ref{lem:counting} with $(\delta')_{\ref{lem:counting}} = (\frac{\delta}{2})^r$ 
implies $|\mathcal{F}| \ge \delta_1 \prod_{i=1}^{r} |V_i|$, and
that $\mathcal{K}$ is $\delta_1$-heavy with respect to $\mathcal{F}$.

Apply Lemma~\ref{lem:rolling_drc_general} with 
$(F_0)_{\ref{lem:rolling_drc_general}} = K_r$,
$(\mathcal{K})_{\ref{lem:rolling_drc_general}} = \mathcal{K}$,
$(\mathcal{F})_{\ref{lem:rolling_drc_general}} = \mathcal{F}$, and
$c_{\ref{lem:rolling_drc_general}} = r-1$
to obtain families $\mathcal{B}$ and $\mathcal{T}$ (the parameters are chosen
so the conditions of Lemma~\ref{lem:rolling_drc_general} is satisfied).
Then Properties (i), (iii) of Lemma~\ref{lem:rolling_drc_general} 
immediately implies Properties (i), (ii), (iv) of Lemma~\ref{lem:rolling_drc_degenerate_g}
(recall that if $X$ has negligible $(0,d,2\beta)$-potential in $Y$, then $X$ is $(d,2\beta)$-common
into $Y$).
Property (ii) of Lemma~\ref{lem:rolling_drc_general} implies
\begin{align*}
	|\mathcal{F}(\mathcal{B})| 
	\ge \left(\frac{\delta_1}{\log^2 n} \right)^{(10s)^{r}} |\mathcal{F}|
	\ge& \left(\frac{\delta_1}{\log^2 n} \right)^{(10s)^{r}} \delta_1 \prod_{i=1}^{r} |V_i| \\
	=& \left(\Big(\frac{\delta^2}{2}\Big)^{r^2} \frac{1}{\log^2 n} \right)^{(10s)^{r}} \Big(\frac{\delta^2}{2}\Big)^{r^2} \prod_{i=1}^{r} |V_i|
	\ge \Big(\frac{\delta}{\log n}\Big)^{2r^2(10s)^{r}} \prod_{i=1}^{r} |V_i|\,,
\end{align*}
which is Property (iii) of Lemma~\ref{lem:rolling_drc_degenerate_g}.
\end{proof}

\subsubsection{Proof of Lemma \ref{lem:rolling_drc_general}}

Let $G$ be an $(\varepsilon,\delta)$-dense $r$-partite graph with vertex partition $(V_i)_{i \in [r]}$.
Let $F_0$ be an $r$-partite graph and $\mathcal{F}$ be a family of copies of $F_0$ in $G$. 
Suppose that an $r$-tuple of sets $\mathcal{A} = (A_i)_{i \in [r]}$, and a family
$\mathcal{K}$ of copies of $K_r$ across $\mathcal{A}$ satisfying
$|\mathcal{K}| \ge \delta_2 \prod_{i \in [r]} |V_i|$ is given.
For families $\mathcal{X}$ and $\mathcal{Y}$ of subgraphs,
we define $\rho(\mathcal{X}, \mathcal{Y})$ as the number
of pairs $(X, Y) \in \mathcal{X} \times \mathcal{Y}$ for which $X$ and $Y$ are adjacent.

To prove the lemma, we will iteratively construct sets $T_i$ for $i \in [r]$.
In the $t$-th step of our process, we construct a set $T_t$ so that
the $t$-tuple $\mathcal{T}_t = (T_i)_{i \in [t]}$ has the following properties.
Define $T_{-i, t} = \bigcup_{j \in [t] \setminus \{i\}} T_j$ for all $i \in [r]$. 
Define $A_{i,t} = A_i \cap N(T_{-i, t})$ for $i \in [r]$ and
$\mathcal{A}_t = (A_{i,t})_{i \in [r]}$. 
There exists a non-empty family $\mathcal{K}_t$ of copies of $K_{r-t}$ across
$(A_{i,t})_{i = t+1}^{r}$ satisfying
\begin{itemize}
  \setlength{\itemsep}{1pt} \setlength{\parskip}{0pt}
  \setlength{\parsep}{0pt}
  \item[(a)] $T_i \subseteq N(T_{i'})$ for all distinct $i,i' \in [t]$,
  \item[(b)] $|\mathcal{K}_t| \ge \left(\frac{\delta_1 \delta_2}{\log^2 n}\right)^{(10s)^{t}} \prod_{i > t} |V_i|$,
\end{itemize}
for $B_{i,t} := V_i \cap N(T_{-i,t})$, $\mathcal{B}_t := (B_{i,t})_{i \in [r]}$ and $B_{-i,t} := \bigcup_{j \in [r] \setminus \{i\}} B_{j,t}$,
\begin{itemize}
  \setlength{\itemsep}{1pt} \setlength{\parskip}{0pt}
  \setlength{\parsep}{0pt}
  \item[(c)] $B_{-i,t}$ has $\lambda$-negligible $(cs + (r-t)s, d, 2\beta)$-potential in $B_{i,t}$ for all $i \le t$, and
  \item[(d)]  $\rho(\mathcal{K}_t, \mathcal{F}(\mathcal{B}_t)) \ge \left(\frac{\delta_1}{\log^2 n}\right)^{(10s)^{t}} |\mathcal{K}_t| |\mathcal{F}|$.
\end{itemize}
Furthermore, if $\mathcal{A}$ has $\lambda$-negligible $((r-1)s, d, \beta)$-potential, then
\begin{itemize}
  \setlength{\itemsep}{1pt} \setlength{\parskip}{0pt}
  \setlength{\parsep}{0pt}
  \item[(e)] $A_{-i}$ has $\lambda$-negligible $((r-t)s, d, \beta)$-potential in $A_{i,t}$ for $i \le t$, and
  $A_{-i}$ has $\lambda$-negligible $((r-t-1)s, d, \beta)$-potential in $A_{i,t}$ for $i > t$.
\end{itemize}
See Figure \ref{fig:fig_general}.

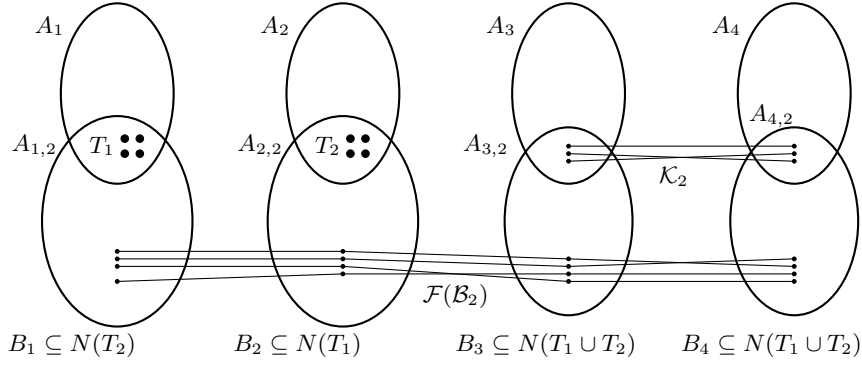
\begin{figure}[h]
  \centering
  \begin{tabular}{ccc}

\begin{tikzpicture}
  \clip (-1.5,-2.5) rectangle (12.5,3.3);

  \draw [thick] (1,1.5) ellipse (0.75cm and 1.2cm);
  \draw (0.1,2.4) node {\footnotesize{$A_1$}};

  \draw [thick] (1,-0.2) ellipse (1.0cm and 1.4cm);
  \draw (0.4,-1.85) node {\footnotesize{$B_1 \subseteq N(T_2)$}};

  \draw (-0.1,0.8) node {\footnotesize{$A_{1,2}$}};

  \draw [fill=black] (1.1,0.9) circle (0.5mm);
  \draw [fill=black] (1.1,0.7) circle (0.5mm);
  \draw [fill=black] (1.3,0.9) circle (0.5mm);
  \draw [fill=black] (1.3,0.7) circle (0.5mm);
  \draw (0.8,0.8) node {\footnotesize{$T_1$}};

  \draw [thick] (4,1.5) ellipse (0.75cm and 1.2cm);
  \draw (3.1,2.4) node {\footnotesize{$A_2$}};

  \draw [thick] (4,-0.2) ellipse (1.0cm and 1.4cm);
  \draw (3.4,-1.85) node {\footnotesize{$B_2 \subseteq N(T_1)$}};

  \draw (2.9,0.8) node {\footnotesize{$A_{2,2}$}};

  \draw [fill=black] (4.1,0.9) circle (0.5mm);
  \draw [fill=black] (4.1,0.7) circle (0.5mm);
  \draw [fill=black] (4.3,0.9) circle (0.5mm);
  \draw [fill=black] (4.3,0.7) circle (0.5mm);
  \draw (3.8,0.8) node {\footnotesize{$T_2$}};

  \draw [thick] (7,1.5) ellipse (0.75cm and 1.2cm);
  \draw (6.1,2.4) node {\footnotesize{$A_3$}};

  \draw [thick] (7,-0.2) ellipse (0.85cm and 1.25cm);
  \draw (6.7,-1.85) node {\footnotesize{$B_3 \subseteq N(T_1 \cup T_2)$}};

  \draw (5.9,0.8) node {\footnotesize{$A_{3,2}$}};

  \draw [thick] (10,1.5) ellipse (0.75cm and 1.2cm);
  \draw (9.1,2.4) node {\footnotesize{$A_4$}};

  \draw [thick] (10,-0.2) ellipse (0.85cm and 1.25cm);
  \draw (9.7,-1.85) node {\footnotesize{$B_4 \subseteq N(T_1 \cup T_2)$}};

  \draw (9.7,1.2) node {\footnotesize{$A_{4,2}$}};

  \draw [fill=black] (7,0.8) circle (0.25mm);
  \draw [fill=black] (10,0.8) circle (0.25mm);
  \draw (7,0.8) -- (10,0.8);

  \draw [fill=black] (7,0.7) circle (0.25mm);
  \draw [fill=black] (10,0.6) circle (0.25mm);
  \draw (7,0.7) -- (10,0.6);

  \draw [fill=black] (7,0.6) circle (0.25mm);
  \draw [fill=black] (10,0.7) circle (0.25mm);
  \draw (7,0.6) -- (10,0.7);

  \draw (8.4,0.4) node {\footnotesize{$\mathcal{K}_2$}};  

  \draw [fill=black] (1,-0.7) circle (0.25mm);
  \draw [fill=black] (4,-0.7) circle (0.25mm);
  \draw [fill=black] (7,-0.7) circle (0.25mm);
  \draw [fill=black] (10,-0.7) circle (0.25mm);
  \draw (1,-0.7) -- (4,-0.7) -- (7,-0.8) -- (10,-0.7);

  \draw [fill=black] (1,-0.6) circle (0.25mm);
  \draw [fill=black] (4,-0.6) circle (0.25mm);
  \draw [fill=black] (7,-0.8) circle (0.25mm);
  \draw [fill=black] (10,-0.8) circle (0.25mm);
  \draw (1,-0.6) -- (4,-0.6) -- (7,-0.7) -- (10,-0.8);

  \draw [fill=black] (1,-0.8) circle (0.25mm);
  \draw [fill=black] (4,-0.8) circle (0.25mm);
  \draw [fill=black] (7,-1.0) circle (0.25mm);
  \draw [fill=black] (10,-1.0) circle (0.25mm);
  \draw (1,-0.8) -- (4,-0.8) -- (7,-1.0) -- (10,-1.0);

  \draw [fill=black] (1,-1.0) circle (0.25mm);
  \draw [fill=black] (4,-0.9) circle (0.25mm);
  \draw [fill=black] (7,-0.9) circle (0.25mm);
  \draw [fill=black] (10,-0.9) circle (0.25mm);
  \draw (1,-1.0) -- (4,-0.9) -- (7,-0.9) -- (10,-0.9);

  \draw (5.5,-1.2) node {\footnotesize{$\mathcal{F}(\mathcal{B}_2)$}};






\end{tikzpicture}
  \end{tabular}
  \caption{The families $\mathcal{A}$, $\mathcal{A}_t$, $\mathcal{T}_t$, $\mathcal{B}_t$,  $\mathcal{K}_t$, and $\mathcal{F}(\mathcal{B}_t)$ for $t=2$ and $r=4$. }
\label{fig:fig_general}
\end{figure}

We first verify that for $t = r$, an $r$-tuple $\mathcal{T}_r$ satisfying the
above listed properties satisfies the claimed properties of the lemma.
Note that Property (ii) of Lemma \ref{lem:rolling_drc_general} follows from (d) 
since $|\mathcal{K}_r| = 1$ (since $\mathcal{K}_r$ is non-empty, it consists of 
the unique copy of the empty graph).
Property (i) follows from (c) and Proposition~\ref{prop:potential_monotone}. 
Moreover, if $\mathcal{A}$ has $\lambda$-negligible $((r-1)s, d, \beta)$-potential, then
Property (iii) follows from (e).

For the initial case $t=0$, Properties (a), (c), and (e) are trivially true.
Furthermore for $\mathcal{K}_0 := \mathcal{K}$, Property (b) follows from the given 
bound on $|\mathcal{K}|$ and Property (d) follows from the fact that 
$\mathcal{K}$ is $\delta_1$-heavy with respect to $\mathcal{F}$
since $\mathcal{B}_0 = (V_i)_{i \in [r]}$.

Now suppose that we have successfully completed the $t$-the step to construct
a $t$-tuple of sets $\mathcal{T}_t = (T_i)_{i \in [t]}$ for some $t \le r-1$.
Consider an auxiliary bipartite graph $\Gamma$ whose vertex 
set is $\mathcal{K}_t \cup \mathcal{F}(\mathcal{B}_t)$
where a pair $(K, F)$ with $K \in \mathcal{K}_t$ and $F \in \mathcal{F}(\mathcal{B}_t)$ forms an edge
if they are adjacent, i.e., if the $i$-th part of $K$ and the $i'$-th part of $F$ forms a complete bipartite
graph for each pair of distinct indices $i, i' \in [r]$.
Property (d) implies that the number of edges of $\Gamma$ is at least
$\left(\frac{\delta_1}{\log^2 n}\right)^{10s^{t}}|\mathcal{K}_t||\mathcal{F}|$.
Hence there are at least 
$\frac{1}{2}\left(\frac{\delta_1}{\log^2 n}\right)^{10s^{t}}|\mathcal{K}_t||\mathcal{F}|$ edges
incident to vertices $K \in \mathcal{K}_t$ of degree at least
$\frac{1}{2} \left(\frac{\delta_1}{\log^2 n}\right)^{10s^{t}}|\mathcal{F}|$ in $\Gamma$.
For $i \ge 0$, define $\mathcal{K}_t^{i}$ as the family of graphs $K \in \mathcal{K}_t$
with degree at least $2^{i-1} \left(\frac{\delta_1}{\log^2 n}\right)^{10s^{t}}|\mathcal{F}|$ 
and less than $2^{i} \left(\frac{\delta_1}{\log^2 n}\right)^{10s^{t}}|\mathcal{F}|$ 
in $\Gamma$.
Since $|\mathcal{F}| \le n^{f}$, there are at most $f\log n$ 
non-empty families $\mathcal{K}_t^i$.
In particular, there exists an index $i_0$ for which the subgraph of $\Gamma$
induced on $\mathcal{K}_t^{i_0} \cup \mathcal{F}$ contains at least
$\frac{1}{2f \log n} \left(\frac{\delta_1}{\log ^2 n}\right)^{(10s)^{t}}|\mathcal{F}|
\ge \left(\frac{\delta_1}{\log ^2 n}\right)^{2(10s)^{t}}|\mathcal{F}|$ edges.
Define $\alpha := 2^{i_0}$ and $\mathcal{K}_t' := \mathcal{K}_t^{i_0}$.
Thus
\begin{align} \label{eq:alpha}
	\alpha |\mathcal{K}_t'|
	\ge \left(\frac{\delta_1}{\log^2 n}\right)^{2(10s)^{t}}|\mathcal{K}_t||\mathcal{F}|.	
\end{align}
Moreover, since $\alpha \le |\mathcal{F}|$, it follows that
\begin{align} \label{eq:alpha2}
	|\mathcal{K}_t'|
	\ge \left(\frac{\delta_1}{\log^2 n}\right)^{2(10s)^{t}}|\mathcal{K}_t|
	\ge  \left(\frac{\delta_1}{\log^2 n}\right)^{2(10s)^{t}} \cdot \left(\frac{\delta_1 \delta_2}{\log^2 n}\right)^{(10s)^{t}} \prod_{i > t} |V_i|,
\end{align}
where the second inequality follows from Property (b).

Let $\hat{T}_{t+1}$ be a $s$-tuple of vertices in $A_{t+1,t}$ chosen 
uniformly and independently at random. 
Define $\hat{\mathcal{T}}_{t+1}$ as the $(t+1)$-tuple of sets obtained from 
$\mathcal{T}_{t}$ by adding the set $\hat{T}_{t+1}$.
Define $\hat{\mathcal{A}}_{t+1} = (\hat{A}_{i,t+1})_{i \in [r]}$ and 
$\hat{\mathcal{B}}_{t+1} = (\hat{B}_{i,t+1})_{i \in [r]}$ as above using $\hat{\mathcal{T}}_{t+1}$.
Define $\hat{\mathcal{K}}_{t+1}$ as the copies $K$ of $K_{r-t-1}$ for which $V(K) \cup \{x\}$ forms a 
copy of $K_{r-t}$ in $\mathcal{K}_t'$ for every $x \in \hat{T}_{t+1}$.

As in the previous section, we will consider $(p,d,\beta)$-potentials for fixed $d$ and two 
different values of $\beta$.
Hence for simplicity, we slightly abuse notation, and for sets $X, Y$ and an integer $p$, 
denote the $(p,d,\beta)$-potential of $X$ in $Y$ as $\xi_p(X,Y)$ and
the $(p,d,2\beta)$-potential of $X$ in $Y$ as $\eta_p(X,Y)$.
Define
\[
	\xi(\hat{\mathcal{A}}_{t+1})
	:= \sum_{i \le t} \frac{\xi_{(r-t-1)s}(A_{-i}, \hat{A}_{i,t+1})}{\xi_{(r-t)s}(A_{-i}, A_{i,t}) }
	+ \sum_{i \ge t+2} \frac{\xi_{(r-t-2)s}(A_{-i}, \hat{A}_{i,t+1})}{\xi_{(r-t-1)s} (A_{-i}, A_{i,t})},
\]
and
\begin{align*}
	\eta(\hat{\mathcal{B}}_{t+1})
	:=&\, 
	\frac{\eta_{cs + (r-t-1)s}(\hat{B}_{-(t+1),t+1}, B_{t+1,t})}{\lambda^{cs+(r-t)s-1}}
	\,+ \sum_{i \le t} \frac{\eta_{cs+(r-t-1)s}(B_{-i,t}, \hat{B}_{i,t+1})}{\eta_{cs+(r-t)s}(B_{-i,t}, B_{i,t})}.
\end{align*}
The formulas are designed so that $\xi(\hat{\mathcal{A}}_{t+1}) < \lambda^{-s}$ implies Property (e), 
and $\eta(\hat{\mathcal{B}}_{t+1}) < \lambda^{-s}$ implies Property (c).
Define
\[
	\mu := 
	\BBE\left[ \rho\left(\hat{\mathcal{K}}_{t+1}, \mathcal{F}(\hat{\mathcal{B}}_{t+1}) \right)
- \left(\frac{\delta_1}{\log^2 n}\right)^{(10s)^{t+1}} |\hat{\mathcal{K}}_{t+1}||\mathcal{F}|
- \lambda^s |\mathcal{K}||\mathcal{F}|\left(\xi(\hat{\mathcal{A}}_{t+1}) + \eta(\hat{\mathcal{B}}_{t+1})\right) \right],
\]
and let $T_{t+1}$ be a particular choice of $\hat{T}_{t+1}$ for which the random variable on the right-hand side becomes at least its expected value. 
Similarly, define the non-hat versions of parameters such as $\mathcal{A}_{t+1}$ 
as the family $\hat{\mathcal{A}}_{t+1}$ for this particular choice of $\hat{T}_{t+1}$.
Property (a) immediately follows from our choice.
The other properties can be verified using the following claim.

\begin{claim} \label{clm:drclaim_general}
$\mu \ge \left(\frac{\delta_1 \delta_2}{\log^2 n}\right)^{(10s)^{t+1}} \left(\prod_{i > t} |V_i|\right) |\mathcal{F}|$.
\end{claim}

Given this claim, we see that
\[
	\rho\left(\mathcal{K}_{t+1}, \mathcal{F}(\mathcal{B}_{t+1}) \right)
	- \left(\frac{\delta_1}{\log^2 n}\right)^{(10s)^{t+1}} |\mathcal{K}_{t+1}| |\mathcal{F}|
	\ge 0,
\]
from which Property (d) follows.
Also since $\rho\left(\mathcal{K}_{t+1}, \mathcal{F}(\mathcal{B}_{t+1}) \right)  \le |\mathcal{K}_{t+1}| |\mathcal{F}|$, we see that
\[
	|\mathcal{K}_{t+1}| |\mathcal{F}|
	\ge \mu
	\ge \left(\frac{\delta_1 \delta_2}{\log^2 n}\right)^{(10s)^{t + 1}} \left(\prod_{i > t} |V_i|\right) |\mathcal{F}|,
\]
from which Property (b) follows.

Furthermore, since $\mu > 0$, we see that
\[
	\lambda^s |\mathcal{K}||\mathcal{F}| \cdot \left(\xi(\mathcal{A}_{t+1}) + \eta(\mathcal{B}_{t+1})\right)
	< \rho\left(\mathcal{K}, \mathcal{F}(\mathcal{B}_{t+1}) \right)
	\le |\mathcal{K}||\mathcal{F}|,
\]
from which it follows that $\xi(\mathcal{A}_{t+1}) + \eta(\mathcal{B}_{t+1}) < \lambda^{-s}$.
Since $\eta(\mathcal{B}_{t+1}) < \lambda^{-s}$, it immediately follows that
$\eta_{cs+(r-t-1)s}(B_{-(t+1),t+1}, B_{t+1,t}) < \lambda^{cs+(r-t-1)s-1}$, thus verifying
Property (c) for the $(t+1)$-th step for $i= t+1$ since $B_{t+1,t+1} = B_{t+1,t}$.
For $i \le t$, we see that
\[
	\eta_{cs + (r-t-1)s}(B_{-i,t}, B_{i,t+1}) 
	< \lambda^{-s} \eta_{cs + (r-t)s}(B_{-i,t}, B_{i,t})
	< \lambda^{-s} \lambda^{cs + (r-t)s - 1} = \lambda^{cs + (r-t-1)s - 1},
\]
where the second inequality follows from Property (c) for the $t$-th step.
Since $B_{-i,t+1} \subseteq B_{-i,t}$ it follows that
$\eta_{cs + (r-t-1)s}(B_{-i,t+1}, B_{i,t+1}) \le \eta_{cs + (r-t-1)s}(B_{-i,t}, B_{i,t+1})$
and thus Property (c) for the $(t+1)$-th step holds. Property (e) follows similarly.

\begin{proof}[Proof of Claim~\ref{clm:drclaim_general}]
In order to compute $\mu$, we start with computing $\BBE[|\hat{\mathcal{K}}_{t+1}|]$.
We then compute 
$\BBE\left[\rho\left(\hat{\mathcal{K}}_{t+1}, \mathcal{F}(\hat{\mathcal{B}}_{t+1}) \right)\right]$ 
in terms of $\BBE[|\hat{\mathcal{K}}_{t+1}|]$, 
and finish by computing $\BBE\left[\xi(\hat{\mathcal{A}}_{t+1}) + \eta(\hat{\mathcal{B}}_{t+1}) \right]$.

Let $\mathcal{L}$ be the family of copies of $K_{r-t-1}$ across $(A_{i,t})_{i \in [t+2,r]}$.
By linearity of expectation,
\begin{align} \label{eq:kt1_1}
	\BBE\left[ |\hat{\mathcal{K}}_{t+1}| \right]
	= \sum_{K \in \mathcal{L}} \BFP( K \in \hat{\mathcal{K}}_{t+1} )
	= \sum_{K \in \mathcal{L}} \left(\frac{d(K)}{|A_{t+1,t}|} \right)^{s},
\end{align}
where $d(K)$ denotes the number of vertices $a \in A_{t+1, t}$ which together with $K$
forms a copy of $K_{r-t}$ in $\mathcal{K}_t'$. Then by convexity,
$\BBE\left[ |\hat{\mathcal{K}}_{t+1}| \right]
	\ge |\mathcal{L}| \left(\frac{1}{|\mathcal{L}|} \sum_{K \in \mathcal{L}} \frac{d(K)}{|A_{t+1,t}|} \right)^{s}$.
Hence by $\sum_{K \in \mathcal{L}}d(K) = |\mathcal{K}_{t}'|$ and \eqref{eq:alpha2}, 
\begin{align} \label{eq:kt1_2}
	\BBE\left[ |\hat{\mathcal{K}}_{t+1}| \right]
	\ge&\, |\mathcal{L}| \left( \left(\frac{\delta_1}{\log^2 n}\right)^{2(10s)^t}
	\left(\frac{\delta_1\delta_2}{\log^2 n}\right)^{(10s)^t} \right)^{s}
	\ge |\mathcal{L}| \left(\frac{\delta_1\delta_2}{\log^2 n}\right)^{3s \cdot (10s)^{t}}.
\end{align}

Next step is to compute  $\BBE\left[\rho\left(\hat{\mathcal{K}}_{t+1}, \mathcal{F}(\hat{\mathcal{B}}_{t+1}) \right)\right]$. 
For simplicity, denote $\mathcal{F}_t = \mathcal{F}(\mathcal{B}_t)$.
For $K \in \mathcal{L}$ and $F \in \mathcal{F}_t$, define $d(K,F)$ as the number of vertices $a \in A_{t+1,t}$
for which $\{a\} \cup V(K)$ forms a copy of $K_{r-t} \in \mathcal{K}_t'$ that is adjacent to $F$.
By linearity of expectation and convexity,
\[
	 \BBE\left[\rho\left(\hat{\mathcal{K}}_{t+1}, \mathcal{F}_t(\hat{\mathcal{B}}_{t+1}) \right) \right]
	 = \sum_{K \in \mathcal{L}} \sum_{F \in \mathcal{F}_t} \left(\frac{d(K,F)}{|A_{t+1,t}|}\right)^{s}
	 \ge \sum_{K \in \mathcal{L}} |\mathcal{F}_t| \left(\frac{1}{|\mathcal{F}_t|}\sum_{F \in \mathcal{F}_t} \frac{d(K,F)}{|A_{t+1,t}|}\right)^{s}.
\]
By the definition of $\mathcal{K}_t'$ and $\alpha$, for each fixed $K \in \mathcal{L}$ we have
$\sum_{F \in \mathcal{F}_t} d(K,F) \ge \frac{1}{2}\alpha \cdot d(K)$.
Hence
\[
	 \BBE\left[ \rho\left(\hat{\mathcal{K}}_{t+1}, \mathcal{F}(\hat{\mathcal{B}}_{t+1}) \right) \right]
	 \ge \sum_{K \in \mathcal{L}} |\mathcal{F}_t| \left(\frac{\alpha d(K)}{2|\mathcal{F}_t| |A_{t+1,t}|}\right)^{s}
	 =\frac{1}{2^s}\left(\frac{\alpha^s}{|\mathcal{F}_t|^{s-1}}\right) \sum_{K \in \mathcal{L}}  \left(\frac{d(K)}{|A_{t+1,t}|}\right)^{s},
\]
which by \eqref{eq:kt1_1} and \eqref{eq:alpha} gives
\begin{align*}
	\BBE\left[\rho\left(\hat{\mathcal{K}}_{t+1}, \mathcal{F}(\hat{\mathcal{B}}_{t+1}) \right) \right]
	\ge&\, \left(\frac{\alpha^s}{2^s |\mathcal{F}_t|^{s-1}}\right) \BBE\left[ |\hat{\mathcal{K}}_{t+1}| \right]
	\ge \BBE\left[ |\hat{\mathcal{K}}_{t+1}| \right] \cdot 
	\frac{1}{2^s}\left(\frac{\delta_1}{\log^2 n}\right)^{2s(10s)^{t}} |\mathcal{F}_t|.
\end{align*}
Since $\rho(\mathcal{K}_t, \mathcal{F}_t) \le |\mathcal{K}_t| |\mathcal{F}_t|$, 
Property (d) implies $|\mathcal{F}_t| \ge \left(\frac{\delta_1}{\log^2 n}\right)^{(10s)^{t}}|\mathcal{F}|$, and thus
\begin{align*}
	\BBE\left[\rho\left(\hat{\mathcal{K}}_{t+1}, \mathcal{F}(\hat{\mathcal{B}}_{t+1}) \right) \right]
	\ge&\, \BBE\left[ |\hat{\mathcal{K}}_{t+1}| \right] \cdot 
	\frac{1}{2^s} \left(\frac{\delta_1}{\log^2 n}\right)^{2s(10s)^{t}} \left(\frac{\delta_1}{\log^2 n}\right)^{(10s)^{t}} |\mathcal{F}|.	
\end{align*}
Therefore
\begin{align*}
	\BBE\left[ \rho\left(\hat{\mathcal{K}}_{t+1}, \mathcal{F}(\hat{\mathcal{B}}_{t+1}) \right) 
	 - \left(\frac{\delta_1}{\log^2 n}\right)^{(10s)^{t+1}} |\hat{\mathcal{K}}_{t+1}||\mathcal{F}| \right]
	\ge&\, \BBE\left[ \left(\frac{\delta_1}{\log^2 n}\right)^{3s(10s)^{t}} |\hat{\mathcal{K}}_{t+1}||\mathcal{F}|\right] \\
	\ge &\, \left(\frac{\delta_1\delta_2}{\log^2 n}\right)^{6s (10s)^{t}} |\mathcal{L}||\mathcal{F}|,
\end{align*}
where the second inequality follows from \eqref{eq:kt1_2}.
From \eqref{eq:alpha2}, 
$|\mathcal{L}| \ge \frac{1}{|V_{t+1}|} |\mathcal{K}_t'| \ge \left(\frac{\delta_1\delta_2}{\log^2 n}\right)^{3 (10s)^{t}} \prod_{i > t+1} |V_i|$,
and hence
\begin{align} \label{eq:almost_final}
	\BBE\left[ \rho\left(\hat{\mathcal{K}}_{t+1}, \mathcal{F}(\hat{\mathcal{B}}_{t+1}) \right) 
	 - \left(\frac{\delta_1}{\log^2 n}\right)^{(10s)^{t+1}} |\hat{\mathcal{K}}_{t+1}||\mathcal{F}| \right]
	\ge&\, \left(\frac{\delta_1\delta_2}{\log^2 n}\right)^{9s (10s)^{t}} |\mathcal{F}| \prod_{i > t+1} |V_i|.
\end{align}

To compute $\BBE\left[\xi(\hat{\mathcal{A}}_{t+1}) + \eta(\hat{\mathcal{B}}_{t+1}) \right]$, we compute
the expectation of each summand of $\xi(\hat{\mathcal{A}}_{t+1})$
and $\eta(\hat{\mathcal{B}}_{t+1})$ using Proposition~\ref{prop:negligible_potential}.
For instance for all $i \le t$, Proposition~\ref{prop:negligible_potential}~(ii) implies that
\[
	\BBE\left[\frac{\eta_{cs + (r-t-1)s}(B_{-i,t}, \hat{B}_{i,t+1})}{\eta_{cs + (r-t)s}(B_{-i,t}, B_{i,t})}\right]
	\le \frac{1}{|A_{t+1,t}|^s}.
\]
Since $|\mathcal{K}_t| \le \prod_{i = t+1}^{r} |A_{i,t}|$, Property~(b) 
implies $|A_{t+1,t}| \ge \left(\frac{\delta_1 \delta_2}{\log^2 n}\right)^{(10s)^t}|V_{t+1}| \ge \sqrt{\lambda n}$.
Thus 
\[
	\BBE\left[\frac{\eta_{(r-t-1)s}(B_{-i,t}, \hat{B}_{i,t+1})}{\eta_{(r-t)s}(B_{-i,t}, B_{i,t})}\right]
	\le \frac{1}{(\lambda n)^{s/2}}
	< \frac{1}{2r n^{r+f} \lambda^s}.
\]
One can similarly bound the other terms using Proposition~\ref{prop:negligible_potential}~(ii) 
except for the term
\begin{align} \label{eq:exception}
	\BBE\left[\frac{\eta_{cs + (r-t-1)s}(\hat{B}_{-(t+1),t+1}, B_{t+1,t})}{\lambda^{cs+(r-t)s-1}}\right].
\end{align}
Let $s_t = cs + (r-t-1)s$. For this term we use Proposition~\ref{prop:negligible_potential}~(i) to obtain
\begin{align*}
	\BBE\left[\eta_{s_t}(\hat{B}_{-(t+1),t+1}, B_{t+1,t}) \right]
	\le&\, \eta_{s_t}(B_{-(t+1),t+1}, B_{t+1,t}) \cdot \left(\frac{2\beta}{|A_{t+1,t}|}\right)^{s} 
	\le\, n^{d + s_t}\left(\frac{2\beta}{|A_{t+1,t}|}\right)^{s},
\end{align*}
where the second inequality follows from the trivial bound that $\eta_{s_t}(X,Y)$ is 
at most the number of $(s_t+d)$-tuples in $X$.
Since $\frac{\beta}{n} \le \left(\frac{\lambda}{n}\right)^{c+r+1}$
and $|A_{t+1,t}| \ge 2\lambda$, the above gives
\begin{align*}
	\BBE\left[\eta_{s_t}(\hat{B}_{-(t+1),t+1}, B_{t+1,t}) \right]
	\le&\, n^{d + s_t }\left( \frac{\lambda}{n}\right)^{(c+r)s}
	\le \frac{1}{2r n^{r+f}} \cdot \lambda^{s_t - 1},
\end{align*}
where the second inequality holds since it is equivalent to
$ (\frac{n}{\lambda})^{(c+r)s  - s_t} \ge 2r \lambda n^{d+r+f}$, which holds
since $s_t \le (c+r-1)s$ and $(\frac{n}{\lambda})^{s} \ge 2r n^{d+r+f+1}$.
Therefore
\[
	\BBE\left[\frac{\eta_{s_t}(\hat{B}_{-(t+1),t+1}, B_{t+1,t})}{\lambda^{s_t+s-1}}\right] \le \frac{1}{2rn^{r+f}\lambda^s}.
\] 
Hence
\[
	\BBE\left[\xi(\hat{\mathcal{A}}_{t+1}) + \eta(\hat{\mathcal{B}}_{t+1}) \right] 
	\le 2r \cdot \frac{1}{2rn^{r+f}\lambda^s} \le \frac{1}{n^{r+f} \lambda^s}
\]
Since $|\mathcal{K}| \le n^r$ and $|\mathcal{F}| \le n^f$, it follows from \eqref{eq:almost_final}
that
\begin{align*}
	&\, \BBE\left[ \rho\left(\hat{\mathcal{K}}_{t+1}, \mathcal{F}(\hat{\mathcal{B}}_{t+1}) \right)
- \left(\frac{\delta_1}{\log^2 n}\right)^{(10s)^{t+1}} |\hat{\mathcal{K}}_{t+1}||\mathcal{F}|
- \lambda^s |\mathcal{K}||\mathcal{F}|\left(\xi(\hat{\mathcal{A}}_{t+1}) + \eta(\hat{\mathcal{B}}_{t+1})\right) \right] \\
	\ge &\, \left(\frac{\delta_1\delta_2}{\log^2 n}\right)^{(10s)^{t+1}} |\mathcal{F}| \prod_{i > t+1} |V_i|. \qedhere
\end{align*}
\end{proof}

\section{Applications} \label{sec:application}

In this section, we apply the tools developed in the previous sections to
problems in extremal graph theory and Ramsey theory.
The known techniques to these problems based on the blow-up lemma
cannot be extended directly using our version of the blow-up lemma 
due to the lack of a constrained version.
We overcome this difficulty by invoking Lemma~\ref{lem:rolling_drc_degenerate_g} instead.

The main embedding lemma of this section is inspired by (and adapts) the techniques
developed by B\"ottcher, Schacht, and Taraz in their proof of the bandwidth theorem.
Let $B_k^r$ be the $kr$-vertex graph obtained from a path on $k$ vertices
by replacing every vertex by a clique of size $r$ and replacing every edge by
a complete bipartite graph minus a perfect matching.
More precisely, the vertex set of $B_k^r$ is $[k] \times [r]$ and two vertices
$(i,j)$, $(i', j')$ are adjacent if and only if 
$|i - i'| \le 1$ and $j \neq j'$.

Let $G$ be a graph and $(V_i)_{i \in I}$ be a family of disjoint subsets of vertices indexed by
elements from some set $I$. We define the {\em $(\varepsilon, \delta)$-reduced graph of $(V_i)_{i \in I}$}
as the graph on $I$ where two vertices $i,j \in I$ are adjacent if and only if the pair
$(V_i, V_j)$ is $(\varepsilon,\delta)$-dense.
A family of vertex subsets $(V_{i,j})_{i \in [k], j \in [r+1]}$ forms 
an {\em $(\varepsilon,\delta)$-backbone} of a graph if it 
contains $B_k^r$ as a labelled subgraph over $[k] \times [r]$, 
and for each $i \in [k]$, the subgraph (of the reduced graph) induced on 
$\{i \} \times [r+1]$ forms a $K_{r+1}$. \footnote{the backbone defined here is slightly different from that used in \cite{BoScTa}.} 
As can be seen in the following lemma, backbone structures 
turn out to be useful in embedding graphs.

\begin{lem} \label{lem:backbone_embedding}
For each fixed $\varepsilon, \delta, r$  satisfying $\varepsilon \le (\frac{\delta}{2})^{4r}$, there exists $c$ such that the following holds.
For $\beta \le e^{-c(d\log n)^{(4r-1)/4r} (\log \log n)^{1/4r}} n$,
let $G$ be an $n$-vertex graph with an $(\varepsilon^2,\delta)$-backbone $(V_{i,j})_{i \in [k+1], j \in [r]}$
where $|V_{i,j}| \ge \frac{\varepsilon}{kr}n$ holds for all $i \in [k+2], j \in [r]$.
Let $H$ be a graph with a $d$-degenerate $\beta$-local labelling $[m]$ and a vertex partition
$(W_{i,j})_{i \in [k], j \in [r]}$ satisfying the following:
\begin{itemize}
  \setlength{\itemsep}{1pt} \setlength{\parskip}{0pt}
  \setlength{\parsep}{0pt}
\item[(i)] For each $i \in [k]$, we have $\bigcup_{j \in [r]} W_{i,j} = ((i-1)\xi, i\xi] \cap [m]$, 
\item[(ii)] For each $i \in [k]$, we have $W_{i, r} \subseteq ((i-1)\xi + \beta, i \xi - \beta]$
\item[(iii)] for each $j \in [r]$, the set $\bigcup_{i \in [k]}W_{i,j}$ is independent,
\item[(iv)] $|W_{i,j}| \le (1-\varepsilon)|V_{i,j}|$ for all $(i,j) \in [k] \times [r]$.
\end{itemize}
Then $G$ contains a copy of $H$.
\end{lem}
\begin{proof}
Define $n_0 = \frac{\varepsilon^2}{kr} n$, $s = (\frac{d\log n}{\log \log n})^{1/4r}$, and  $\lambda = (\frac{\delta}{\log n})^{10 (10s)^{4r-1}} n$.
Define $\delta_1 = (\frac{\delta}{2})^{8r^2}$ and $\delta_2 = \varepsilon^{2r}\left(\frac{\delta_1}{\log^2n}\right)^{2(10s)^{2r}}$.
Since $\varepsilon, \delta, r$ are fixed, if $c$ is large enough depending on these parameters, then
one can check that the condition of  Lemma~\ref{lem:rolling_drc_general} is satisfied when $f_{\ref{lem:rolling_drc_general}} \le 2r$,
$r_{\ref{lem:rolling_drc_general}} \le 2r$, $c_{\ref{lem:rolling_drc_general}} \le 2r-1$, and
$n$ is sufficiently large. Define $\eta = (\frac{\delta}{2})^{2r}$.

For each $i \in [k]$ and $j \in [r]$, 
make the bipartite graph between $V_{i,j}$ and $V_{i+1,j}$ complete
by adding edges if necessary. Furthermore, for all $j \in [r-1]$, 
make the bipartite graph between $V_{i,r}$ and $V_{i+1,j}$ complete (for $i \in [k-1]$), and 
the bipartite graph between $V_{i,r}$ and $V_{i-1,j}$ complete (for $i \in [2,k]$).
If we can find an embedding of $H$ to this graph with $W_{i,j}$ 
mapping to $V_{i,j}$ for all $(i,j) \in [k] \times [r]$, 
then it also forms an embedding to the original graph $G$ by Properties (ii) and (iii).
Thus by abusing notation, we let $G$ denote this new graph.
Note that for each $i \in [k-1]$, the $2r$-vertex 
$(\varepsilon^2,\delta)$-reduced graph over the parts
$\bigcup_{j \in [r]} V_{i,j} \cup V_{i+1,j}$ forms a complete graph.

For each $i \in [k]$, define $V_i := \bigcup_{j \in [r]} V_{i,j}$, 
and $W_i := \bigcup_{j \in [r]} W_{i,j} = ((i-1)\xi, i\xi]$ (by Property (i)).
For simplicity, we assume that $m = k\xi$ and
for each $i \in [k]$, partition the interval $W_i =  ((i-1)\xi, i\xi]$ into $T = \frac{\xi}{\beta}$ intervals 
$W^{(1)}_i, \cdots, W^{(T)}_i$ of equal lengths. For each $(i,j) \in [k]\times[r]$ and $t \in [T]$, 
define $W^{(t)}_{i,j} \subseteq W^{(t)}_i$ as the subset of vertices of color $j$.
For each $i \in [k]$ and real number $x$, define $\mathcal{K}^{(i)}_{x}$ as the family 
of copies of $K_r$ across $(V_{i,j})_{j \in [r]}$ that are
$x$-heavy with respect to the $2r$-part family $\bigcup_{j \in [r]} V_{i,j} \cup V_{i+1,j}$.
By (a slight modification of) Lemma~\ref{lem:counting}, we see that
$\mathcal{K}^{(i)}_{\eta}$ is $\delta_1$-heavy with respect to $\mathcal{K}^{(i)}_{2\eta}$,
and $|\mathcal{K}^{(i)}_{\eta}| \ge |\mathcal{K}^{(i)}_{2\eta}| \ge \delta_1 \prod_{j \in [r]} |V_{i,j}|$ for all $i \in [k]$.

Our proof is based on two layers of iterative processes. 
For $s \ge 1$, the $s$-th step of the outer-layer 
takes as input a partial embedding $f$ defined on $(\bigcup_{i \le s-1} W_i) \cup W^{(1)}_s$
and a family of sets $\mathcal{A}_s = (A_{j})_{j \in [r]}$ satisfying the following conditions:
\begin{itemize}
  \setlength{\itemsep}{1pt} \setlength{\parskip}{0pt}
  \setlength{\parsep}{0pt}
\item[(a)] $f(W_{i,j}) \subseteq V_{i,j}$ for all $i \le s-1$ and $j \in [r]$,
\item[(b)] $f(W_{s,j}^{(1)}) \subseteq A_j \subseteq V_{s,j}$ for each $j \in [r]$,
\item[(c)] $\mathcal{A}_s$ has $\lambda$-negligible $\left((r-1)s, d, 2\beta\right)$-potential, and
\item[(d)] $|\mathcal{K}^{(s)}_{\eta}(\mathcal{A}_s)| \ge \delta_2 \prod_{i=1}^{r} |V_{s,j}|$.
\end{itemize}
It then outputs an extension of $f$ to $(\bigcup_{i \le s} W_i) \cup W^{(1)}_{s+1}$ and constructs
a family of sets $\mathcal{A}_{s+1} = (A'_{j})_{j \in [r]}$ satisfying (a)-(d) listed
above for the $(s+1)$-th step.
Thus after the $k$-th step, we obtain an embedding of $H$ into $G$.

As the base case ($s=0$) of the outer-layer, 
apply Lemma~\ref{lem:rolling_drc_general} with 
$\mathcal{A}_{\ref{lem:rolling_drc_general}} = (V_{1,j})_{j \in [r]}$,
$\mathcal{F}_{\ref{lem:rolling_drc_general}} = \mathcal{K}_{\ref{lem:rolling_drc_general}} = \mathcal{K}^{(1)}_{2\eta}$,
$c_{\ref{lem:rolling_drc_general}} = r-1$,
$(\delta_1)_{\ref{lem:rolling_drc_general}} = \delta_1$,
$(\delta_2)_{\ref{lem:rolling_drc_general}} = \delta_2$, and
$\beta_{\ref{lem:rolling_drc_general}} = \beta$,
to obtain a family
$\mathcal{A}_1 = (A_j)_{j \in [r]}$ where (1) $A_j \subseteq V_{1,j}$ for all $j \in [r]$, 
(2) $\mathcal{A}_1$ is $(d,2\beta)$-common, 
(3) $\mathcal{A}_1$ has $\lambda$-negligible $((r-1)s, d, 2\beta)$-potential, and
(4) $|\mathcal{K}^{(1)}_{2\eta}(\mathcal{A}_1)| \ge \delta_2 |\mathcal{K}^{(1)}_{2\eta}|$. 
Apply Lemma~\ref{lem:rolling_embedding_g} to find a partial embedding $f$ of $H$ defined on $W^{(1)}_1$.
Note that this satisfies Properties (a)-(d).

Now suppose that we are at the $s$-th step of the outer-layer process for some $s \ge 1$.
Define $U_j = V_{s,j}$ and $U_{r+j} = V_{s+1,j}$ for $j \in [r]$, and
consider the $2r$-part family $\mathcal{U} = (U_j)_{j \in [2r]}$.
Recall that the $(\varepsilon^2,\delta)$-reduced graph of $\mathcal{U}$ is complete.
Define $I^{(i)} = W^{(i)}_{s}$ for $i \in [T]$, and $I^{(T+1)} := W^{(1)}_{s+1}$. 
Let $H_s$ be the subgraph of $H$ induced on $\bigcup_{i \le T+1} I^{(i)}  = W_s \cup W^{(1)}_{s+1}$.
Color $H_s$ with $2r$-colors so that the vertices in $\bigcup_{i \le T} I^{(i)}$
keep the same color as in $H$, and each vertex of color $j$ in $I^{(T+1)}$ 
gets re-colored with color $r+j$ for all $j \in [r]$.
For each $i \in [T+1]$ and $j \in [2r]$, 
define $I^{(i)}_j$ as the set of vertices in $I^{(i)}$ of color $j$.
Let $F_0$ be the complete $r$-partite graph with $2$ vertices in each part.
Note that $F_0$ can be considered as a $2r$-partite graph with $1$ vertex in each part as well 
(we will be alternating between the two viewpoints).
For each positive real number $x$, define $\mathcal{F}_x$ as the family of copies of $F_0$
across $\mathcal{U}$ (as a $2r$-partite graph)
that are $x$-heavy with respect to the $3r$-part family
$\bigcup_{j \in [r]} V_{s,j} \cup V_{s+1,j} \cup V_{s+2,j}$.
By (a slight modification of) Lemma~\ref{lem:counting}, we see that 
$\mathcal{F}_{\eta}$ is $\delta_1$-heavy with respect to $\mathcal{F}_{2\eta}$
and $|\mathcal{F}_{2\eta}| \ge \delta_1 \prod_{j \in [2r]} |U_j|$.

The $t$-th step of the inner-layer takes as input a partial embedding $f$ of $H_s$ 
defined on $\bigcup_{i \le t} I^{(i)}$
and a family of sets $\mathcal{B}_t = (B_{j})_{j \in [2r]}$ satisfying 
\begin{itemize}
  \setlength{\itemsep}{1pt} \setlength{\parskip}{0pt}
  \setlength{\parsep}{0pt}
\item[(A)] $f(I^{(t)}_j) \subseteq B_j \subseteq U_j \setminus f\left(\bigcup_{i \le t-1} I^{(i)}_j\right)$ for each $j \in [2r]$,
\item[(B)] $\mathcal{B}_t$ has $\lambda$-negligible $\left(s', d, 2\beta\right)$-potential for all $s' \le (2r-1)s$, and
\item[(C)] $|\mathcal{F}_{\eta}(\mathcal{B}_t; G_t)| \ge \delta_2 \prod_{j \in [2r]} |U_j|$,
\end{itemize}
where $G_t$ is the subgraph of $G$ induced on $(\bigcup_{j \in [2r]} U_j) \setminus f\left(\bigcup_{i \le t-1} I^{(i)}_j\right)$ and $\mathcal{F}_{\eta}(\mathcal{B}_t; G_t)$ is the family of $\eta$-heavy (in $G_t$)
copies of $F_0$ across $\mathcal{B}_t$.
It then extends $f$ to $\bigcup_{i \le t+1} I^{(i)}$
and constructs  a family of sets $\mathcal{B}_{t+1}$ satisfying properties
(A), (B), and (C) given above for the $(t+1)$-th step.

Suppose that we successfully terminated the $T$-th step of the inner-layer to extend $f$
to $\bigcup_{i \le T+1} I^{(i)} = W_s \cup W^{(1)}_{s+1}$
and obtain a family of sets $\mathcal{C} = (C_{j})_{j \in [2r]}$.
Define $\mathcal{A}_{s+1} = (C_{j+r})_{j \in [r]}$.
For each $j \in [r]$, since $W_{s,j} = \bigcup_{i \le T} I^{(i)}_j$ and $W^{(1)}_{s+1,j} = I^{(T+1)}_j$, 
by Property (A) and the definition of $U_j$, it follows that
$f(W_{s,j}) \subseteq V_{s,j}$ and $f(W^{(1)}_{s+1,j}) \subseteq C_{j+r} \subseteq V_{s+1,j}$.
This proves Properties (a) and (b) for the $(s+1)$-th step of the outer-layer process.
Property (c) follows from Property (B).
Note that for each copy of $F_0$ in $\mathcal{F}_{\eta}(\mathcal{C})$, 
its subgraph $K_r$ over the $r$ vertices intersecting $\mathcal{A}_{s+1}$
is in $\mathcal{K}^{(s+1)}_{\eta}$. Thus from Property (C) we have
\begin{align*}
	|\mathcal{K}_{\eta}^{(s+1)}(\mathcal{A}_{s+1})| 
	\ge \frac{1}{\prod_{j \in [r]} |C_{j}|} |\mathcal{F}_{\eta}(\mathcal{C}; \mathcal{G}_{T+1})|
	&\ge \delta_2 \prod_{j = r+1}^{2r} |U_j| 
	= \delta_2 \prod_{j = r+1}^{2r} |V_{s+1,j}|,
\end{align*}
and Property (d) follows. Hence successful termination of the inner-layer 
terminates the $s$-th step of the outer-layer.

For the base case ($t=1$) of the inner-layer, note that
Lemma~\ref{lem:counting} implies
$\mathcal{K}^{(s)}_{\eta}$ being $\delta_1$-heavy with respect to $\mathcal{F}_{2\eta}$.
Hence Lemma~\ref{lem:rolling_drc_general} applied with 
$r_{\ref{lem:rolling_drc_general}} = r$,
$(V_j)_{\ref{lem:rolling_drc_general}} = U_j \cup U_{r+j}$ for all $j \in [r]$,
$\mathcal{A}_{\ref{lem:rolling_drc_general}} = \mathcal{A}_s$,
$\mathcal{K}_{\ref{lem:rolling_drc_general}} = \mathcal{K}^{(s)}_{\eta}$,
$\mathcal{F}_{\ref{lem:rolling_drc_general}} = \mathcal{F}_{2\eta}$,
$c_{\ref{lem:rolling_drc_general}} = 2r-1$,
$(\delta_1)_{\ref{lem:rolling_drc_general}} = \delta_1$,
$(\delta_2)_{\ref{lem:rolling_drc_general}} = \delta_2$,
$(\beta)_{\ref{lem:rolling_drc_general}} = 2\beta$
provides families
$\mathcal{T} = (T_j)_{j \in [r]}$ and $\mathcal{X} = (X_j)_{j \in [r]}$ satisfying 
(1) $X_j \subseteq U_j \cup U_{r+j}$ for all $j \in [r]$, 
(2) $\mathcal{X}$ is $(d,4\beta)$-common, 
(3) $\mathcal{X}$ has $\lambda$-negligible $(s', d, 4\beta)$-potential
for all $s' \le (2r-1)s$, and
(4) $|\mathcal{F}_{2\eta}(\mathcal{X})| \ge (\frac{\delta_1}{\log^2n})^{(10s)^{2r}}|\mathcal{F}_{2\eta}|$.
Apply Lemma~\ref{lem:rolling_embedding_g} to extend $f$ to $I^{(2)}$ so that 
$f(I^{(2)}_j) \subseteq A_j \cap X_j$ for all $j \in [r]$.
For each $j \in [r]$, define $B_j = (U_j \cap X_j) \setminus f(I^{(1)}_j)$ and $B_{r+j} = U_{r+j} \cap X_j$. Let $\mathcal{B} = (B_j)_{j \in [2r]}$.
Since $G_2$ is obtained from $G_1$ by removing $\beta$ vertices, 
every $2\eta$-heavy copy of $F_0$ in $G_1$ is $\eta$-heavy in $G_2$.
Then since  $|\mathcal{F}_{2\eta}| \ge \delta_1 \prod_{j \in [2r]} |U_j|$,
\[
	|\mathcal{F}_{\eta}(\mathcal{B}; G_2)|
	\ge |\mathcal{F}_{2\eta}(\mathcal{X}; G_1)| - \beta n^{2r-1}
	\ge \frac{1}{2}\left(\frac{\delta_1}{\log^2n}\right)^{(10s)^{2r}}|\mathcal{F}_{2\eta}|
	\ge \delta_2 \prod_{j \in [2r]} |U_j|.
\]
Thus Properties (A), (B), and (C) holds for $t=2$.
The general case of the inner-layer can be done along the line of the
proof of Theorem~\ref{thm:main_extend}. We omit the details.
\end{proof}

A typical application of Lemma~\ref{lem:backbone_embedding} requires two steps:
first, finding a backbone structure, and
second, finding a coloring and a labelling of the vertices of $H$ with certain properties.
We use the regularity lemma to find a backbone structure.
A vertex partition $V = V_0 \cup V_1 \cup \cdots \cup V_k$ of a given $n$-vertex graph
is {\it $\varepsilon$-regular} if $|V_0| \le \varepsilon n$, $V_i$ has equal sizes for all $i \ge 1$, 
and $(V_i, V_j)$ is $\varepsilon$-regular for all but at most $\varepsilon k^2$ pairs of indices 
$i,j \in [k]$.

\begin{lem} \label{lem:regularity}
For all $\varepsilon$ and $t$, there exists $n_0$ and $T$ such that every $n$-vertex graph with $n \ge n_0$
admits an $\varepsilon$-regular partition into $k+1$ parts $(V_i)_{i =0}^{k}$ for some $k \in [t,T]$
where for each index $i \in [k]$, there are at most $\varepsilon k$ other indices $j \in [k]$
for which $(V_i, V_j)$ is not $\varepsilon$-regular.
\end{lem}

In order to prepare $H$ for embedding, we apply the following lemma which
can be seen as an extension of Lemma~\ref{lem:bandwidth_degenerate}.
Since the results in \cite{BoScTa} implicitly implies the lemma, we provide its proof in the appendix.

\begin{lem} \label{lem:lemmaH}
For all $\varepsilon$, there exists a positive constant $c$ such that the following holds.
If $k$ is a positive integer and
$H$ is a $d$-degenerate $r$-colorable $m$-vertex graph with bandwidth at most $\beta$
satisfying $m \ge ckr \beta$, 
then there exists an $(r+1)$-coloring by $[r+1]$ and a 
$5d$-degenerate $\beta \log_2(4\beta)$-local labelling by $[m]$
that satisfies the following properties. For $i \in [k]$ and $j \in [r+1]$,
define $W_{i,j}$ as the set of vertices of color $j$ in the interval 
$(i\frac{m}{r}, (i+1)\frac{m}{k}]$.
\begin{itemize}
  \setlength{\itemsep}{1pt} \setlength{\parskip}{0pt}
  \setlength{\parsep}{0pt}
  \item[(i)] For all $i \in [k]$ and $j \in [r]$, we have $|W_{i,j}| \le (1+\varepsilon)\frac{m}{kr}$,
  \item[(ii)] for all $i \in [k]$, we have $|W_{i,r+1}| \le \varepsilon \frac{m}{kr}$, and
  \item[(iii)] for all $i \in [k]$,we have $W_{i,r+1} \cap (i\frac{m}{k} - \beta \log_2(4\beta), i\frac{m}{k}+\beta \log_2(4\beta)] = \emptyset$.
\end{itemize}
\end{lem}

\subsection{Bandwidth theorem for degenerate graphs}

The following theorem is simple corollary of the bandwidth theorem.

\begin{lem} \label{lem:backbone}
For all positive integer $r$ and positive reals $\varepsilon$ and $\delta$, 
there exists $n_0$ such that the following holds
for all $n \ge n_0$.  If $G$ is a $n$-vertex graph of minimum degree at 
least $(1- \frac{1}{r} + \delta)n$, then  
$G$ contains $B_{\lfloor n/r \rfloor}^{r}$ as a subgraph.
\end{lem}

The following lemma prepares the graph $G$ for Theorem~\ref{thm:main_bandwidth}.

\begin{lem} \label{lem:lemmaG}
For all $\delta, \varepsilon, t_0$ and $r$ satisfying $\varepsilon \le \max\{\frac{\delta}{4}, \frac{1}{4}\}$
and $r \ge 2$, there exists $n_0$ and $T$ such that the following holds for all $n \ge n_0$.
Let $G$ be an $n$-vertex graph with minimum degree at least $(1- \frac{1}{r} + 2\delta)n$.
Then there exists an $(\varepsilon,\delta)$-backbone $(V_{i,j})_{i \in [k], j \in [r+1]}$
for $k \in [t_0, T]$, satisfying $|V_{i,j}| \ge (1- \varepsilon)\frac{n}{kr}$ for all $(i,j) \in [k] \times [r]$,
and $|V_{i,r+1}| \ge \frac{\varepsilon^2}{k r} n$ for all $i \in [k]$.
\end{lem}
\begin{proof}
Apply Lemma~\ref{lem:regularity} with $\varepsilon_{\ref{lem:regularity}} = \varepsilon^3$ 
and $t_{\ref{lem:regularity}} = \max\{2rt_0, \frac{2}{\varepsilon^3}, (n_0)_{\ref{lem:backbone}}(\varepsilon^3, \delta)\}$
to find an $\varepsilon^3$-regular partition $V = \bigcup_{i =0}^{t} V_i$ of $G$ for some 
$\frac{2}{\varepsilon^3} \le t \le T_{\ref{lem:regularity}}$ (for simplicity, assume that $t$
is divisible by $r$).
Denote $m = |V_1|$.
For each $i \in [t]$, there are at most $\varepsilon^3 t$ other indices $j$ for which
$(V_i, V_j)$ is not $\varepsilon^3$-regular.
If there are $\alpha t$ other indices $j$ for which the bipartite graph
between $V_i$ and $V_j$ has density less than $\delta + \varepsilon$, then
by the minimum degree condition of $G$,
\[
	m \cdot \left(1 - \frac{1}{r} + 2\delta\right)n
	\le
	e(V_i, V) \le 2e(V_i) + \alpha t \cdot (\delta + \varepsilon) m^2 + (1-\alpha)t \cdot m^2.
\]
Since $2e(V_i) \le m^2$ and $m \le \frac{n}{t} \le \frac{\varepsilon^3}{2}n$, the above implies
$1 - \frac{1}{r} + 2\delta \le \frac{\varepsilon^3}{2} + 1- (1-\delta - \varepsilon)\alpha$, 
from which it follows that
$\alpha < \frac{1}{r} - \delta + \varepsilon^3$.
Therefore the $(\varepsilon^3,\delta+\varepsilon)$-reduced graph 
$R$ over $(V_i)_{i \in [t]}$ has minimum degree at least $(1 - \frac{1}{r} + \delta)t$.
By Lemma~\ref{lem:backbone}, we can re-label the indices so that
the parts are labelled by $[k] \times [r]$, 
and the $(\varepsilon^3,\delta)$-reduced graph of 
$(V_{i,j})_{i \in [k], j \in [r]}$ forms a copy of $B_k^r$.

For each $i \in [k]$ the number of indices $(i',j') \in [k] \times [r]$ that are adjacent to 
$(i,j)$ for all $j \in [r]$ is, by the minimum degree condition of $R$, at least $\delta r t$.
Fix one such index $a(i)$ for each $i \in [k]$ so that each $(i',j') \in [k] \times [r]$
is selected by at most $\frac{1}{\delta}$ indices (this can be done greedily since
$\frac{1}{\delta} \cdot \delta r t > k$).
For each $i \in [k]$, choose a set $U_{i,r+1}$ as a subset of $V_{a(i)}$
of size $\frac{\delta \varepsilon}{2}|V_{a(i)}|$ so that the sets $U_{i,r+1}$ are disjoint
for distinct $i$. Since $\delta \ge 4\varepsilon$, it follows that $|U_{i,r+1}| \ge \frac{\varepsilon^2}{kr}n$
for all $i \in [k]$.
Then for each $(i,j) \in [k] \times [r]$, define $U_{i,j}$ as a subset of $V_{i,j}$
obtained by removing the sets $U_{i',r+1}$ having $a(i') = (i,j)$.
Note that for all $(i,j) \in [k] \times [r]$,
\[
	|U_{i,j}| 
	\ge |V_{i,j}| - \frac{1}{\delta} \frac{\delta \varepsilon}{2}|V_{i,j}|
	\ge \left(1 - \frac{\varepsilon}{2}\right) (1- \varepsilon^3) \frac{n}{t}
	\ge \left(1 - \frac{\varepsilon}{2}\right) (1- \varepsilon^3)^2 \frac{n}{kr}
	\ge (1-\varepsilon) \frac{n}{kr}.
\]
Define $a(i,r+1) = a(i)$ for $i \in [k]$ and $a(i,j) = (i,j)$ for $(i,j) \in [k] \times [r]$. 
Note that for each $(i,j) \in [k] \times [r+1]$, we have $|U_{i,j}| \ge \varepsilon^2 |V_{a(i,j)}|$. 
Therefore $(U_{i,j}, U_{i',j'})$ is $(\varepsilon, \delta)$-dense whenever
$(V_{a(i,j)}, V_{a(i',j')})$ is $(\varepsilon^3, \delta)$-dense.
It follows that the family $(V_{i,j})_{i \in [k], j \in [r+1]}$ forms an $(\varepsilon,\delta)$-backbone.
\end{proof}

The proof of Theorem~\ref{thm:main_bandwidth} easily follows. 
We restate the theorem here with a refined bound on the bandwidth condition.

\begin{thm} \label{thm:main_bandwidth_refine}
For all integers $r, d$ and real numbers $\varepsilon, \delta$,
there exists $c$ such that the following holds for all sufficiently large $n$.
If $G$ is an $n$-vertex graph of minimum degree at least $(1 - \frac{1}{r} + \delta)n$,
and $H$ is an $r$-partite $d$-degenerate graph of bandwidth at most 
$e^{-c(d\log n)^{(4r+3)/(4r+4)} (\log \log n)^{1/(4r+4)}}n$
on at most $(1-\varepsilon)n$ vertices, then $G$ contains $H$ as a subgraph.
\end{thm}
\begin{proof}[Proof of Theorem~\ref{thm:main_bandwidth}]
We may assume that $\varepsilon < (\frac{\delta}{4})^{4(r+1)}$
by reducing its value of necessary.
Define $t_0 = \frac{3}{\varepsilon^2}$
and $T = T_{\ref{lem:lemmaG}}(\frac{\delta}{2}, t_0, r, \frac{\varepsilon}{3})$.
Let $c = c_{\ref{lem:backbone_embedding}}(\frac{\varepsilon}{2}, \frac{\delta}{2}, r+1)$, and $n$ be a large enough integer.

Let $G$ be an $n$-vertex graph of minimum degree at least $(1 - \frac{1}{r} + \delta)n$. 
By Lemma~\ref{lem:lemmaG} with $\delta_{\ref{lem:lemmaG}} = \frac{\delta}{2}$
and $\varepsilon_{\ref{lem:lemmaG}} = \frac{\varepsilon^2}{4}$, 
there exists an $(\frac{\varepsilon^2}{4},\frac{\delta}{2})$-backbone 
$(V_{i,j})_{i \in [k], j \in [r+1]}$ satisfying 
$|V_{i,j}| \ge (1- \frac{\varepsilon^2}{4})\frac{n}{kr}$ for all $(i,j) \in [k] \times [r]$,
and $|V_{i,r+1}| \ge \frac{\varepsilon^4}{16k r} n$ for all $i \in [k]$
for some positive integer $k \in [t_0, T]$.
Let $H$ be a $m$-vertex graph and apply Lemma~\ref{lem:lemmaH} to $H$ 
with $\varepsilon_{\ref{lem:lemmaH}} = \frac{\varepsilon^4}{32}$ and $k_{\ref{lem:lemmaH}} = k-2$,  
to obtain a $5d$-degenerate $\beta \log_2(4\beta)$-local labelling by $[m]$
that satisfies the following properties. For $i \in [k]$ and $j \in [r+1]$,
define $W_{i,j}$ as the set of vertices of color $j$ in the interval 
$(i\frac{m}{k-2}, (i+1)\frac{m}{k-2}]$.
\begin{itemize}
  \setlength{\itemsep}{1pt} \setlength{\parskip}{0pt}
  \setlength{\parsep}{0pt}
  \item[(i)] For all $i \in [k-2]$ and $j \in [r]$, we have $|W_{i,j}| \le (1+\frac{\varepsilon^4}{32})\frac{m}{(k-2)r}$,
  \item[(ii)] for all $i \in [k-2]$, we have $|W_{i,r+1}| \le \frac{\varepsilon^4}{32} \frac{m}{(k-2)r}$, and
  \item[(iii)] for all $i \in [k-2]$, there is no vertex of color $r+1$ in the interval $(i\frac{m}{(k-2)} - \beta \log_2(4\beta), i\frac{m}{(k-2)}+\beta \log_2(4\beta)]$.
\end{itemize}
Since $m \le (1-\varepsilon)n$, for $(i,j) \in [k-2] \times [r]$,
\[
	|W_{i,j}| 
	\le \left(1+\frac{\varepsilon^4}{32}\right)\frac{m}{(k-2)r}  
	\le \left(1-\frac{\varepsilon}{2}\right)|V_{i,j}|,
\]
and similarly $|W_{i,r+1}| \le \frac{\varepsilon^4}{32} \frac{m}{(k-2)r} \le (1-\frac{\varepsilon}{2}) |V_{i,r+1}|$ for all $i \in [k-2]$.
We can now apply Lemma~\ref{lem:backbone_embedding} 
with $\delta_{\ref{lem:backbone_embedding}} = \frac{\delta}{2}$, $\varepsilon_{\ref{lem:backbone_embedding}} = \frac{\varepsilon}{2}$, and
$r_{\ref{lem:backbone_embedding}} = r+1$ to find an embedding of $H$ to $G$.
\end{proof}

\subsection{Ramsey theory}

The $r$-th power of a path on $k$ vertices, which we denote by $P_k^r$,
is a graph on $[k]$ for which two vertices $i,i' \in [k]$ are adjacent
if and only if $|i - i'| \le r$.
The following lemma was proved in \cite{AlBrSk}.

\begin{lem} \label{lem:robust_ramsey}
There exists $\varepsilon_0$ such that for every $r$,
the following holds for sufficiently large $n$. 
Let $G$ be an $n$-vertex graph of minimum degree at least $(1-\varepsilon_0)n$. 
Then in every edge-coloring of $G$ with two colors, there exists a monochromatic copy of 
$P_k^r$ for $k \ge \lfloor \frac{n}{2r+3} \rfloor$.
\end{lem}

Given a family of sets whose reduced graph contains $P_k^r$, one can modify
the sets to find a family whose reduced graph contains $B_k^{r+1}$.

\begin{lem} \label{lem:path_to_backbone}
Let $G$ be a given graph.
If the $(\varepsilon,\delta)$-reduced graph of a family $(V_i)_{i \in [k]}$ contains
$P_k^r$ as a subgraph, then there exists a family $(U_{i,j})_{i \in [k-r], j \in [r+1]}$
whose $((r+1)\varepsilon,\delta)$-reduced graph contains $B_k^{r+1}$ as a subgraph.
Furthermore if $|V_i| = m$ for all $i \in [k]$, then $|U_{i,j}| = \frac{m}{r+1}$
for all $(i,j) \in [k-r] \times [r+1]$.
\end{lem}
\begin{proof}
For each $(i,j) \in [k-r] \times [r+1]$, define $a(i,j)$ as the unique element in the interval $[i, i+r+1)$
which equals $j$ modulo $r+1$.
For each $(i,j) \in [k-r] \times [r+1]$, choose $U_{i,j} \subseteq V_{a(i,j)}$ so that 
$|U_{i,j}| = \frac{1}{r+1}|V_{a(i)}|$ and the sets are disjoint for distinct pairs $(i,j)$
(this can be done since for each $a \in [k]$, there are at most $r+1$ pairs $(i,j)$ having $a(i,j) = a$).
We omit the details of verifying that the family $(U_{i,j})_{i \in [k-r], j \in [r+1]}$ satisfies the 
claimed properties.
\end{proof}

The proof of Theorem~\ref{thm:ramsey_degenerate} easily follows.
We restate the theorem here with a refined bound on the bandwidth condition.

\begin{thm*}
For all positive integers $r$ and $d$,
there exists $c$ such that the following holds for all sufficiently large $m$.
Let $H$ be an $m$-vertex $r$-chromatic $d$-degenerate graph of bandwidth at most 
$e^{-c(d\log m)^{(4r+3)/(4r+4)} (\log \log m)^{1/(4r+4)}}m$.
Then $r(H) \le (2r+5)m$.
\end{thm*}
\begin{proof}[Proof of Theorem~\ref{thm:ramsey_degenerate}]
Let $\varepsilon_0 = (\varepsilon_0)_{\ref{lem:robust_ramsey}}(r)$
and let $\varepsilon < \varepsilon_0$ and be small enough depending on $r$.
Let $c = c_{\ref{lem:backbone_embedding}}(\varepsilon, \frac{1}{2} - \varepsilon^2, r+1)$.
Furthermore, let $t_0$ be large enough depending on $r$ and $\varepsilon$.

Let $n = (2r+5)m$.
Given an red/blue edge-coloring of $K_n$, let $G$ be the graph consisting of the red edges.
Apply Lemma~\ref{lem:regularity} with 
$\varepsilon_{\ref{lem:regularity}} = \frac{\varepsilon^2}{r+1}$ and $t_{\ref{lem:regularity}} = t_0$ 
to obtain an $\frac{\varepsilon^2}{r+1}$-regular partition
$\bigcup_{i=0}^{t} V_i$ of $G$ with $t \ge t_0$.
Consider the graph $R$ on $[t]$ where two vertices $i,j \in [t]$ are adjacent if and only if 
$(V_i, V_j)$ is $\frac{\varepsilon^2}{r+1}$-regular in $G$.
Color the edges of $R$ as follows: for an edge $\{i,j\}$, 
if the density of red edges between $(V_i,V_j)$ is at least $\frac{1}{2}$ then color $\{i,j\}$ red,
and color it blue otherwise.
By Lemma~\ref{lem:robust_ramsey} with $r_{\ref{lem:robust_ramsey}} = r$, we can find
a monochromatic copy of $P_k^r$ in $R$ for some $k \ge \lfloor \frac{t}{2r+3} \rfloor$;
say that it forms a red copy (if it was a blue copy, then we can consider the complement of $G$).
Since every $(\varepsilon^2, \frac{1}{2})$-regular pair is $(\varepsilon^2, \frac{1}{2} - \varepsilon^2)$-dense, in the graph $G$, 
the $(\varepsilon^2, \frac{1}{2} - \varepsilon^2)$-reduced graph of
the partition $\bigcup_{i=0}^{t} V_i$ contains $P_k^r$ as a subgraph.
Then By Lemma~\ref{lem:path_to_backbone}, we can find a family $(U_{i,j})_{i \in [k-r], j \in [r+1]}$
whose $(\varepsilon^2,\frac{1}{2} - \varepsilon^2)$-reduced graph contains $B_k^{r+1}$ as a subgraph. 
Moreover, $|U_{i,j}| = \frac{1}{r+1}|V_1| \ge \frac{(1-\varepsilon)n}{t(r+1)}$

Apply Lemma~\ref{lem:lemmaH} to $H$ with 
$\varepsilon_{\ref{lem:lemmaH}} = \varepsilon^2$ to find
a $5d$-degenerate $\beta \log_2(4\beta)$-local labelling by $[m]$
that satisfies the following properties. For $i \in [k-r-2]$ and $j \in [r+1]$,
define $W_{i,j}$ as the set of vertices of color $j$ in the interval 
$(i\frac{m}{r}, (i+1)\frac{m}{k}]$.
\begin{itemize}
  \setlength{\itemsep}{1pt} \setlength{\parskip}{0pt}
  \setlength{\parsep}{0pt}
  \item[(i)] For all $i \in [k-r-2]$ and $j \in [r]$, we have $|W_{i,j}| \le (1+\varepsilon^2)\frac{m}{(k-r-2)r}$, and
  \item[(ii)] for all $i \in [k-r-2]$, there is no vertex of color $r+1$ in the interval $(i\frac{m}{k-r-2} - \beta \log_2(4\beta), i\frac{m}{k-r-2}+\beta \log_2(4\beta)]$.
\end{itemize}
Since $m = \frac{n}{2r+5}$ and $k \ge \lfloor \frac{t}{2r+3} \rfloor$, for all $i,j$,
\begin{align*}
	|W_{i,j}|
	\le \frac{(1+\varepsilon^2)m}{(k-r-2)r} 
	&= \frac{(1+\varepsilon^2)n}{(k-r-2)r(2r+5)} \\
	&\le \frac{(1+\varepsilon^2)n}{((2r+5)t/(2r+3)-(r+3)(2r+5))r} 
	\le (1-\varepsilon) \cdot \frac{(1-\varepsilon)n}{t(r+1)},
\end{align*}
where the last inequality holds since $\varepsilon$ is small enough depending on $r$, and
$t$ is large enough depending on $\varepsilon$ and $r$.
Apply Lemma~\ref{lem:backbone_embedding} with $\varepsilon_{\ref{lem:backbone_embedding}} = \varepsilon$
and $\delta_{\ref{lem:backbone_embedding}} = \frac{1}{2} - \varepsilon^2$ to find an embedding of $H$ to $G$.
This gives a monochromatic copy of $H$.
\end{proof}

\section{Remarks} \label{sec:remarks}

There are two drawbacks of Theorem \ref{thm:main}.
First is that it is an almost-spanning instead of a spanning version.
Second is the lack of its constrained version, i.e., a version 
asserting that we can find an embedding of $H$ into $G$ even when the images of 
some vertices have been chosen beforehand.
Because of that we had to develop Lemma~\ref{lem:backbone_embedding}
instead of directly referring to previously developed strategies
based on the constrained version of the blow-up lemma.
It would be interesting to further refine Theorem~\ref{thm:main} in these directions.

The proof of Lemma~\ref{lem:rolling_drc_degenerate_g} proceeds by
iteratively choosing subsets of vertices $T_i \subseteq V_i$ of sizes $|T_i| = s$
for $i = 1,2, \cdots, r$.
Thus in the end, it is equivalent to choosing a complete $r$-partite graph
with $s$ vertices in each part (each $T_i$ forms one part) according to some 
probability distribution.
It might be possible to simplify the proof by finding a simple description of 
this probability distribution. The fact that we relied on a dyadic decomposition at
each step (see equations \eqref{eq:alpha}, \eqref{eq:alpha2}) seems to make this a difficult task.

The embedding strategy used in this paper can be modified to give some new embedding results 
of graphs of bounded maximum degree and small bandwidth.
We will pursue this direction in another paper.

\medskip

\noindent {\bf Acknowledgements}. I thank David Conlon, Jacob Fox, and Benny Sudakov 
for their extremely helpful comments on an earlier version of this paper. 
I also thank Peter Allen and his students Barnaby Roberts, Matthew Jenssen.

\appendix

\section{Lemma for $H$}

In this section, we prove Lemma~\ref{lem:lemmaH}.
Throughout the section we use the notation $x = y \pm \varepsilon$ to indicate that 
the inequality $y - \varepsilon \le x \le y + \varepsilon$ holds.

\begin{lem} \label{lem:recolor}
Suppose that $r \ge 2$ and $\varepsilon m \ge 3r^2 \beta$.
Let $H$ be a properly $r$-colored graph with a $\beta$-local labelling by $[m]$,
and denote its $r$ color classes by $W_1, W_2, \cdots, W_r$.
For every permutation $\sigma$ of $[r]$, there exists a proper coloring of $H$ by $r+1$ colors whose 
color classes $U_1, \cdots, U_{r+1}$ satisfies the following properties:
\begin{itemize}
  \setlength{\itemsep}{1pt} \setlength{\parskip}{0pt}
  \setlength{\parsep}{0pt}
\item[(i)] $W_i \cap [\beta] = U_i \cap [\beta]$ and $W_i \cap (m-\beta,m] = U_i \cap (m-\beta,m]$ for all $i \in [r]$,
\item[(ii)] $|U_{i}| = |W_{\sigma(i)}| \pm \varepsilon m$ for all $i \in [r]$, and
\item[(iii)] $|U_{r+1}| \le \varepsilon m$.
\end{itemize}
\end{lem}
\begin{proof}
Let $\sigma = \tau_t \circ \cdots \circ \tau_1$ where $\tau_i$ for $i \in [t]$ are transpositions
and $t \le {r \choose 2}$.
For each $i =1,2,\cdots, t$, let $\tau_i = (a_i b_i)$ and do the following:
\begin{itemize}
  \setlength{\itemsep}{1pt} \setlength{\parskip}{0pt}
  \setlength{\parsep}{0pt}
\item[(a)] re-color the vertices of color $a_i$ in $(3i\beta, (3i+2)\beta]$ with color $r+1$,
\item[(b)] re-color the vertices of color $a_i$ in $(m - (3i+2)\beta, m - 3i\beta]$ with color $r+1$, 
\item[(c)] re-color the vertices of color $a_i$ in $((3i+2)\beta, m-(3i+2)\beta]$ with color $b_i$, and
\item[(d)] re-color the vertices of color $b_i$ in $((3i+1)\beta, m-(3i+1)\beta]$ with color $a_i$.
\end{itemize}
We claim that the coloring remains to be a proper coloring after each step. 
Note that the vertices of color $r+1$ added at the $i$-th step is an independent set since
the coloring is proper in the beginning.
Furthermore, by how we defined the interval, their labels differ from that of the vertices
which already had color $r+1$ by more than $\beta$. Therefore by the bandwidth condition, 
the vertices of color $r+1$ remains to be an independent set after the $i$-th step.
Among the vertices having color $a_i$ after the $i$-th step,
the vertices whose color became $a_i$ at the $i$-th step is more than $\beta$
away from a vertex which had color $a_i$ before the $i$-th step, and hence by the bandwidth condition,
the vertices of color $a_i$ forms an independent set.
Similarly, the vertices of color $b_i$ forms an independent set.

In the end, let $U_i$ be the vertices of color $i$. 
Property (i) immediately holds since we never re-colored the vertices in $[3\beta]$ and $(m - 3\beta, m]$.
Note that $U_i \cap ((3t+2)\beta, m - (3t-2)\beta] = W_{\sigma(i)} \cap ((3t+2)\beta, m - (3t-2)\beta]$.
Therefore for all $i \in [r]$,
\[
	|U_i| \le |W_{\sigma(i)}| \pm 2(3t+2)\beta \le |W_{\sigma(i)}| \pm \varepsilon m.
\]
Furthermore, $|U_{r+1}| \le 2(3t+2)\beta \le \varepsilon m$.
\end{proof}

\begin{lem} 
For all $\varepsilon$, there exists a positive constant $c$ such that the following holds.
If $k$ is a positive integer and
$H$ is a $d$-degenerate $r$-colorable $m$-vertex graph with bandwidth at most $\beta$
satisfying $m \ge ckr \beta \log_2 \beta$, 
then there exists a proper $(r+1)$-coloring by $[r+1]$ and a 
$5d$-degenerate $\beta \log_2(4\beta)$-local labelling by $[m]$
that satisfies the following properties. For $i \in [k]$ and $j \in [r+1]$,
define $W_{i,j}$ as the set of vertices of color $j$ in the interval 
$(i\frac{m}{r}, (i+1)\frac{m}{k}]$.
\begin{itemize}
  \setlength{\itemsep}{1pt} \setlength{\parskip}{0pt}
  \setlength{\parsep}{0pt}
  \item[(i)] For all $i \in [k]$ and $j \in [r]$, we have $|W_{i,j}| \le (1+\varepsilon)\frac{m}{kr}$,
  \item[(ii)] for all $i \in [k]$, we have $|W_{i,r+1}| \le \varepsilon \frac{m}{kr}$, and
  \item[(iii)] for all $i \in [k]$, there is no vertex of color $r+1$ in the interval $(i\frac{m}{k} - \beta \log_2(4\beta), i\frac{m}{k}+\beta \log_2(4\beta)]$.
\end{itemize}
\end{lem}
\begin{proof}
Let $c$ be a large enough integer and let $\beta_0 = \beta \log_2(4\beta)$.
For simplicity, we assume that $m = k\xi$ for some integer $\xi$.
Take a $5d$-degenerate $\beta_0$-local labelling of $H$ by $[m]$ whose existence is
guaranteed by Lemma~\ref{lem:bandwidth_degenerate}.

We will show how to color the interval $I = [\xi]$ using $(r+1)$-colors so that 
the set of vertices $W_j \subseteq I$ of color $j$ (defined for each $j \in [r+1]$) satisfies
the following properties: (i') for each $j \in [r]$, $|W_j| \le (1+\varepsilon)\frac{\xi}{r}$,
(ii') $|W_{r+1}| \le \varepsilon \frac{\xi}{r}$, and 
(iii') $W_{r+1}$ does not intersect $[\beta_0]$ nor $(\xi - \beta_0,\xi]$.
The lemma then immediately follows since we can apply the same process to each
interval $((i-1)\xi, i\xi]$ for $i=1,2,\cdots,k$.

Define $\beta' = \frac{2}{\varepsilon} \beta_0$. For simplicity, we assume that
$\beta'$ is an integer and that $\xi = \beta'T$ for some integer $T$; this is 
permitted since $\xi = \frac{m}{k} \ge cr\beta_0$ where $c$ is sufficiently large.
To find the claimed coloring, for each $i \in [T]$, define $I_i = ((i-1)\beta', i\beta']$.
Take an arbitrary $r$-coloring of $H$ by $[r]$ and for each $j \in [r]$, define
define $I_{i,j} \subseteq I_i$ as the subset of vertices of color $j$.
For each index $i \in [T]$, let $\hat{\sigma}_i$ be a uniformly chosen  
random permutation of $[r]$, and define the random variable
$\hat{\xi}_j = \sum_{i \in [T]} |I_{i,\hat{\sigma}_i(j)}|$. By linearity of expectation,
for each $j \in [r]$,
\[
	\mathbb{E}\left[\hat{\xi}_j\right]
	= \sum_{i \in [T]} \frac{1}{r} \sum_{j'=1}^{r} |I_{i,j'}|
	= \frac{1}{r} \sum_{i \in [T]} |I_i| = \frac{\xi}{r}.
\]
Since $|I_{i,j}| \le \beta'$ and $\hat{\xi}_j$
is a sum of $T = \frac{\xi}{\beta'} \ge \frac{1}{2}cr\varepsilon$ random variables, 
by Hoeffding's inequality (see, e.g., \cite{McDiarmid}),
the probability that $\hat{\xi}_j > (1+ \frac{\varepsilon}{2})\frac{\xi}{r}$
is at most $e^{-\Omega(\varepsilon n/(\beta_0 k r))} = o(1)$.
In particular, there exists a choice of permutations $(\hat{\sigma}_i)_{i \in [T]}$ 
for which $\hat{\xi}_j \le (1+ \frac{\varepsilon}{2})\frac{\xi}{r}$ for all $j \in [r]$.
Let $(\sigma_i)_{i \in [T]}$ be such permutations and $\xi_j$ be the corresponding values of $\hat{\xi}_j$
for all $j \in [r]$.

For each $i \in [T]$, apply Lemma~\ref{lem:recolor} to $H[I_i]$ 
with $\sigma_{\ref{lem:recolor}} = \sigma_i$
and $\varepsilon_{\ref{lem:recolor}} = \frac{\varepsilon}{4}$ to
re-color the vertices in the interval $I_i$. Let $U_{i,j}$ be the vertices color $j$ in $I_i$
after re-coloring. Then $|U_{i,j}| = |I_{i, \sigma_i(j)}| \pm \frac{\varepsilon}{4} |I_i|$ 
holds for all $(i,j) \in [T] \times [r]$ and thus for a fixed $j \in [r]$,
\[
	\sum_{i \in [T]} |U_{i,j}| 
	= \sum_{i \in [T]} |I_{i, \sigma_i(j)}| \pm \frac{\varepsilon}{4}|I|
	\le \left(1 + \varepsilon\right) \frac{\xi}{r}.
\]
The other properties straightforwardly follow. We omit the details.
\end{proof}

\end{document}